\documentclass[a4paper,11pt, reqno]{amsart}
\usepackage{cite}
\usepackage{comment}
\usepackage{amsmath,bm}  
\title[Coupled Kirchhoff--Choquard system]{Existence of ground state solutions to Kirchhoff--Choquard system in $\mathbb{R}^3$ with constant potentials}
\author[H. Matsuzawa]{Hiroshi Matsuzawa$^{\dag*}$}
\thanks{$^\dag$  Faculty of Science, Kanagawa University, 3-27-1 Rokkakubashi, Kanagawa-ku, Yokohama-city, Kanagawa, 221-8686, Japan. (Email: \texttt{hmatsu@kanagawa-u.ac.jp})}

\thanks{$^*$ Corresponding Author : Hiroshi Matsuzawa}

\date{\today}
\usepackage{amsmath, enumitem, color, mathrsfs, cases, mathtools}
\mathtoolsset{showonlyrefs}
\usepackage{empheq}

\theoremstyle{definition}

\newtheorem*{Claim}{Claim}
\theoremstyle{plain}
\newtheorem{Th}{Theorem}
\newtheorem{Prop}{Proposition}[section]
\newtheorem{Lem}[Prop]{Lemma}
\newtheorem{Cor}[Prop]{Corollary}

\numberwithin{equation}{section}

\newcommand{\vertiii}[1]{{\left\vert\kern-0.25ex\left\vert\kern-0.25ex\left\vert #1 
    \right\vert\kern-0.25ex\right\vert\kern-0.25ex\right\vert}}

\begin{document}
    \begin{abstract}
In this paper, we consider the following linearly coupled Kirchhoff--Choquard system in $\mathbb{R}^3$:
\begin{empheq}[left = \empheqlbrace]{align*}
& -\left(a_1 + b_1\int_{\mathbb{R}^3} |\nabla u|^2\,dx\right)\Delta u + V_1 u = \mu (I_{\alpha} * |u|^p) |u|^{p - 2} u + \lambda v,\ \ x\in\mathbb{R}^3 \\
& -\left(a_2 + b_2\int_{\mathbb{R}^3} |\nabla v|^2\,dx\right)\Delta v + V_2 v = \nu (I_{\alpha} * |v|^q) |v|^{q - 2} v + \lambda u,\ \ x\in\mathbb{R}^3 \\
& u, v \in H^1(\mathbb{R}^3),
\end{empheq}
where $a_1, a_2, b_1, b_2, V_1, V_2$, $\lambda$, $\mu$ and $\nu$ are positive constants. The function $I_{\alpha} : \mathbb{R}^3 \setminus \{0\} \to \mathbb{R}$ denotes the Riesz potential with $\alpha \in (0, 3)$.

We study the existence of positive ground state solutions under the conditions $\frac{3 + \alpha}{3} < p \le q < 3 + \alpha$, or $\frac{3 + \alpha}{3} < p < q = 3 + \alpha$, or $\frac{3 + \alpha}{3} = p < q < 3 + \alpha$.  
Assuming suitable conditions on $V_1$, $V_2$, and $\lambda$, we obtain a ground state solution by employing a variational approach based on the Nehari–Pohozaev manifold, inspired by the works of Ueno (Commun. Pure Appl. Anal. 24 (2025)) and Chen–Liu (J. Math. Anal. 473 (2019)).

In particular, we emphasize that in the upper half critical case $\frac{3 + \alpha}{3} < p < q = 3 + \alpha$ and the lower half critical case $\frac{3 + \alpha}{3} = p < q < 3 + \alpha$, a ground state solution can still be obtained by taking $\mu$ or $\nu$ sufficiently large to control the energy level of the minimization problem. 

To employ the Nehari--Pohozaev manifold we extend a regularity result to the linearly coupled system, which is essential for the validity of the Pohozaev identity.
    \end{abstract}
        \maketitle
    \section{Introduction and main theorem}
    \subsection{Introduction}
In this paper, we consider the following Kirchhoff–Choquard system in $\mathbb{R}^3$:
\begin{align}\label{eq:NKC}
\begin{cases}
-\displaystyle\left(a_1 + b_1\int_{\mathbb{R}^3} |\nabla u|^2\,dx\right)\Delta u + V_1u = \mu(I_{\alpha}*|u|^p)|u|^{p - 2}u + \lambda v, \\
-\displaystyle\left(a_2 + b_2\int_{\mathbb{R}^3} |\nabla v|^2\,dx\right)\Delta v + V_2v = \nu(I_{\alpha}*|v|^q)|v|^{q - 2}v + \lambda u, \\
u, v \in H^1(\mathbb{R}^3),
\end{cases}
\end{align}
where $a_1, a_2, b_1, b_2, \lambda, \mu$ and $\nu$ are positive constants. The potentials $V_i > 0$ $(i=1,2)$ are constants, and $I_{\alpha}:\mathbb{R}^3\setminus\{0\} \to \mathbb{R}$ denotes the Riesz potential defined by
\[
I_{\alpha}(x) = \frac{\Gamma\left(\dfrac{3 - \alpha}{2}\right)}{\Gamma\left(\dfrac{\alpha}{2}\right) \pi^{3/2} 2^{\alpha} |x|^{3 - \alpha}}.
\]
This problem possesses several notable features. First, it involves two different types of nonlocality, both of which are known to arise in various physical phenomena. In the present setting, the nonlocality on the left-hand side originates from a Kirchhoff-type operator, while that on the right-hand side stems from a Choquard-type convolution nonlinearity. The physical background of each nonlocal term will be discussed shortly. Second, we consider a coupled system of equations, whereas previous studies such as \cite{Lu} and \cite{Chen-Liu-2} dealt with a single equation.
Third, we address a case involving \emph{critical} nonlinearity, the precise meaning of which will be clarified later. The main purpose of this paper is to prove the existence of nontrivial, in particular ground state, solutions to system \eqref{eq:NKC}.

Let us briefly recall some previous results on related problems.

When $\lambda = 0$ and the nonlocal term $\mu (I_{\alpha} * |u|^p) |u|^{p-2}u$ is replaced by the local nonlinearity $f(u)$, equation~\eqref{eq:NKC} reduces to the Kirchhoff--Schr\"{o}dinger equation:
\begin{align}\label{eq:Kirchhoff-single}
-\left(a_1 + b_1 \int_{\mathbb{R}^3} |\nabla u|^2\,dx \right) \Delta u + V_1 u = f(u) \quad \text{for } x \in \mathbb{R}^3.
\end{align}
The Kirchhoff-type equation was introduced by Kirchhoff~\cite{Kirchhoff} as a generalization of the classical d'Alembert wave equation, taking into account the free vibrations of elastic strings:
\begin{align}
\rho \frac{\partial^2 u}{\partial t^2} - \left( \frac{P_0}{h} + \frac{E}{2L} \int_0^L \left| \frac{\partial u}{\partial x} \right|^2\,dx \right) \frac{\partial^2 u}{\partial x^2} = 0.
\end{align}
Kirchhoff's model accounts for changes in the length of the string caused by transverse vibrations.

The Kirchhoff-type equation~\eqref{eq:Kirchhoff-single} has attracted considerable attention. In particular, He and Zou~\cite{He-Zou} established the existence and concentration of positive solutions via variational methods on the Nehari manifold, assuming that the nonlinearity $f \in C^1(\mathbb{R})$ satisfies the Ambrosetti--Rabinowitz condition. This condition holds for $f(u) = |u|^{p-2}u$ if and only if $p > 4$.

Li and Ye~\cite{Li-Ye} considered equation~\eqref{eq:Kirchhoff-single} with the nonlinearity $f(u) = |u|^{p-2}u$.  
They proved the existence of a positive ground state solution for the case $3 < p < 6$.  
Their approach is based on a minimization problem over the \emph{Nehari--Pohozaev} manifold, which was introduced by Ruiz~\cite{Ruiz} for the study of the Schr\"{o}dinger--Poisson equation.  

Guo~\cite{Guo} extended the results of Li and Ye~\cite{Li-Ye} by employing the Nehari--Pohozaev manifold, which is slightly different from the one in \cite{Li-Ye}, and proved the existence of a positive ground state solution under weaker assumptions. In particular, the condition $3 < p < 6$ was relaxed to $2 < p < 6$ for $f(u) = |u|^{p-2}u$ in \cite{Guo}

Tang and Chen~\cite{Tang-Chen} further improved the studies of Li and Ye \cite{Li-Ye} and Guo~\cite{Guo}.  
They proved that equation~\eqref{eq:Kirchhoff-single} admits a ground state solution via the Nehari--Pohozaev manifold under suitable conditions on $f$.  
Unlike~\cite{Guo}, their approach does not require the nonlinearity $f$ to be differentiable.  
To show that the minimizer on the Nehari--Pohozaev manifold is actually a critical point of the energy functional, they employed deformation lemma and degree theory.

We recall the study for the coupled Kirchhoff--Schrodinger system:
\begin{align}\label{eq:Kirchhoff-system}
\begin{cases}
\displaystyle -\left(a_1 + b_1\int_{\mathbb{R}^3} |\nabla u|^2\,dx\right)\Delta u + V_1u = \mu|u|^{p - 2}u + \lambda v, \text{ for } x \in \mathbb{R}^3,   \\
\displaystyle -\left(a_2 + b_2\int_{\mathbb{R}^3} |\nabla v|^2\,dx\right)\Delta v + V_2v = \nu|v|^{q - 2}v + \lambda u, \text{ for } x \in \mathbb{R}^3, \\
             u, v \in H^1(\mathbb{R}^3).
        \end{cases}
\end{align}
 L\"{u} and Peng \cite{Lu-Peng} considered problem \eqref{eq:Kirchhoff-system} with $\mu |u|^{p-2}u$ and $\nu|v|^{q-2}v$ replaced by $f(u)$ and $g(v)$, respectively. They assume the Berestycki--Lions type condition on $f$ and $g$. By this approach we can see that when $2<p,q<6$, problem \eqref{eq:Kirchhoff-system}  admits a nontrivial ground state solution. However, \cite{Lu-Peng} does not contain the case where $f$ or $g$ admit critical growth. 

Recently, Ueno~\cite{Tatsuya} studied system~\eqref{eq:Kirchhoff-system} and obtained a ground state solution by using the Nehari--Pohozaev manifold and the approach developed in~\cite{Tang-Chen}.  
Although the exponents are restricted to $3 < p, q \le 6$, Ueno~\cite{Tatsuya} showed that the approach developed in~\cite{Tang-Chen} is applicable to the critical case $3 < p < q = 6$, and can be used to establish the existence of a ground state solution when the parameter $\mu$ is sufficiently large.

When $\lambda=0$ and $a_1=a_2=1$, $b_1=b_2=0$, $V_1=V_2=1$ in \eqref{eq:NKC} is reduced to the following single Choquard equation of the following form on $\mathbb{R}^N$ with $N\in\mathbb{N}$. 
\begin{align}\label{eq:Choquard}
-\Delta u+u=(I_{\alpha}*|u|^p)|u|^{p-2}u\ \ \mbox{in}\ \ \mathbb{R}^N. 
\end{align}
When $N = 3$, $\alpha = 2$, and $p = 2$, equation~\eqref{eq:Choquard-ori} becomes
\begin{align}\label{eq:Choquard-ori}
- \Delta u + u = (I_2 * u^2)\, u \quad \text{in } \mathbb{R}^3,
\end{align}
which is known as the Choquard--Pekar equation. This equation was first introduced by Pekar~\cite{Pekar} in 1954 to model the quantum mechanics of a polaron at rest.  
In 1976, Choquard employed equation~\eqref{eq:Choquard-ori} to describe an electron trapped in its own hole, as part of a certain approximation within the Hartree--Fock theory of a one-component plasma (see~\cite{Lieb}). Penrose \cite{Penrose} proposed \eqref{eq:Choquard-ori} as a model of self-gravitating matter. Lieb \cite{Lieb} and Lions \cite{Lions-Choquard} obtained the existence of solutions for \eqref{eq:Choquard-ori} by using variational methods.

Also \eqref{eq:Choquard} has a variational structure and the corresponding energy functional is given by
\begin{align*}
E(u)=\frac{1}{2}\int_{\mathbb{R}^N}(|\nabla u|^2+u^2)dx-\int_{\mathbb{R}^N}(I_{\alpha}*|u|^p)|u|^pdx. 
\end{align*}
From the Hardy--Littlewood--Sobolev inequality(see Lemma \ref{H-L-S-ineq} below) $E(u)$ is well defined for $u\in H^1(\mathbb{R}^N)$ if and only if $\frac{N+\alpha}{N}\le p\le \frac{N+\alpha}{N-2}$. 

Moroz and Van Schaftingen \cite{Moroz-Schaftingen} proved that if $\frac{N+\alpha}{N}<p<\frac{N+\alpha}{N-2}$ then \eqref{eq:Choquard} admits a nontrivial weak solution $u\in H^1(\mathbb{R}^N)$. Moreover the authors showed that if $p\ge \frac{N+\alpha}{N-2}$ or $p\le\frac{N+\alpha}{N}$ and $u\in H^1(\mathbb{R}^N)\cap L^{\frac{2Np}{N+\alpha}}(\mathbb{R}^N)$ with $\nabla u\in H^1_{\mathrm{loc}}(\mathbb{R}^N)\cap L^{\frac{2Np}{N+\alpha}}_{\rm{loc}}(\mathbb{R}^N)$ is a weak solution to \eqref{eq:Choquard} then $u=0$ via the Pohozaev identity. In this sense, exponents $p=\frac{N+\alpha}{N}$ and $p=\frac{N+\alpha}{N-2}$ are critical exponents. See also \cite{Moroz-Schaftingen-2} for more results on \eqref{eq:Choquard}. 

There are several studies concerning the linearly coupled system of Choquard-type equations:
\begin{align*}
\begin{cases}
-\Delta u + V_1 u = (I_{\alpha} * F(u)) f(u) + \lambda v, \quad x \in \mathbb{R}^N, \\
-\Delta v + V_2 v = (I_{\alpha} * G(v)) g(v) + \lambda u, \quad x \in \mathbb{R}^N,
\end{cases}
\end{align*}
where $f$ and $g$ are continuous functions, and $F$ and $G$ denote their respective primitives.  
We refer the reader to~\cite{Y-A-S-S, Chen-Liu-1, Xu-Ma-Xing} for related results.

There are also several studies on a single Kirchhoff--Choquard equation:
\begin{align}\label{eq:single-CKS}
- \left(a + b \int_{\mathbb{R}^3} |\nabla u|^2 \, dx \right) \Delta u + V u = (I_{\alpha} * |u|^p) |u|^{p-2} u, \quad x \in \mathbb{R}^3.
\end{align}
L\"{u}~\cite{Lu} proved the existence of a ground state solution to~\eqref{eq:single-CKS} by using the Nehari manifold and the concentration-compactness principle when $2 < p < 3 + \alpha$.  
Recently, Chen and Liu~\cite{Chen-Liu-2} extended L\"{u}'s result by employing the Nehari--Pohozaev manifold.  
They proved the existence of a ground state solution to~\eqref{eq:single-CKS} for $\frac{3+\alpha}{3} < p < 3 + \alpha$. For the fractional counterpart of problem~\eqref{eq:single-CKS}, we refer the reader to~\cite{Ambrosio-Isernia-Temperini}.
Although the Kirchhoff--Choquard system studied in this paper does not directly originate from a specific physical model, it naturally combines two important types of nonlocalities, each of which arises in physical contexts as previously explained. Their interaction introduces new mathematical challenges; in particular, the Choquard-type convolution nonlinearity gives rise to an additional critical exponent, further enriching the problem compared to the Kirchhoff–Schr\"{o}dinger system \eqref{eq:Kirchhoff-system} studied in \cite{Lu-Peng} and \cite{Tatsuya}. Although the system may be regarded as a mathematically natural combination of two well-known types of nonlocal operators, it may also serve as a prototype for exploring the interplay between different nonlocal effects, which has been observed in various physical models. We believe that this study represents an important step toward addressing the above mathematical challenges.

Motivated by the above developments, particularly the results of Ueno~\cite{Tatsuya} and Chen–Liu~\cite{Chen-Liu-2},
we establish the existence of a ground state solution to system~\eqref{eq:NKC}.
In particular, we emphasize that the approach developed in~\cite{Tatsuya} is applicable to both the upper half critical case
$\frac{3 + \alpha}{3} < p < q = 3 + \alpha$ and the lower half critical case $\frac{3 + \alpha}{3} = p < q < 3 + \alpha$.

To obtain a ground state solution as a minimizer of a variational problem on the Nehari–Pohozaev manifold, the weak solution must satisfy the corresponding Pohozaev identity.
Although Yang, Albuquerque, Silva, and Silva mention in \cite{Y-A-S-S} that the Pohozaev identity for the linearly coupled Choquard-type system can be derived by adapting known arguments from~\cite{Moroz-Schaftingen-3, J.M.do-O}, they do not provide a rigorous justification of the regularity required for its validity.
To the best of our knowledge, such a regularity result for the coupled Choquard system has not yet been rigorously established in the literature.
As one of the contributions of this paper, we provide a detailed and self-contained proof of the regularity required for the Pohozaev identity to hold.

In this paper, as a first step, we consider the case where the potentials $V_1$ and $V_2$ are positive constants.  
The problem with nonconstant potentials will be discussed in a forthcoming paper \cite{Matsuzawa}, based on the approach developed in~\cite{Matsuzawa-Ueno}.

\subsection{Assumption and Main theorems}
  We introduce the following norm on $H^1(\mathbb{R}^3)$ which is equivalent to the usual norm on $H^1(\mathbb{R}^3)$:
    \[\|w\|_{a_i, V_i} = \left(a_i\int_{\mathbb{R}^3} |\nabla w|^2\,dx + V_i\int_{\mathbb{R}^3} w^2\,dx\right)^{\frac{1}{2}} \quad \mbox{for}\ \ w \in H^1(\mathbb{R}^3).\]
    We assume that $V_1, V_2$ and $\lambda$ satisfy
    \begin{enumerate}[label = $(\mathrm{V})$, resume = V]
        \item $\lambda \le \delta \sqrt{V_1V_2}$ for some $\displaystyle \delta \in \left(0, 1\right)$.   \label{assumption:potential}
    \end{enumerate}
    We set the product space $H \coloneqq H^1(\mathbb{R}^3) \times H^1(\mathbb{R}^3)$. Then $H$ is a Hilbert space and the norm of $H$ is given by
    \[\|(u, v)\|^2 = \|u\|_{a_1, V_1}^2 + \|v\|_{a_2, V_2}^2.\]
    Let us define the energy functional $I: H \to \mathbb{R}$ corresponding to \eqref{eq:NKC} by
\begin{align}\label{eq:energy}
\begin{split}
I(u,v)&=\frac{1}{2}\left(a_1\int_{\mathbb{R}^3}|\nabla u|^2dx+a_2\int_{\mathbb{R}^3}|\nabla v|^2dx\right)+\frac{1}{2}\left(\int_{\mathbb{R}^3}V_1u^2dx+\int_{\mathbb{R}^3}V_2v^2dx\right) \\ 
      &\ \ \ \ +\frac{1}{4}\left\{b_1\left(\int_{\mathbb{R}^3}|\nabla u|^2dx\right)^2+b_2\left(\int_{\mathbb{R}^3}|\nabla v|^2dx\right)^2\right\} \\
      &\ \ \ -\dfrac{\mu}{2p}\int_{\mathbb{R}^3}(I_{\alpha}*|u|^p)|u|^pdx-\frac{\nu}{2q}\int_{\mathbb{R}^3}(I_{\alpha}*|v|^q)|v|^qdx-\lambda\int_{\mathbb{R}^3}uvdx.
\end{split}
\end{align}
It is easy to verify that $(u,v)\in H$ is a weak solution to \eqref{eq:NKC} if and only if $(u,v)\in H$ is a critical point of $I$, that is, $(u,v)$ satisfies
\begin{align*}
\langle I'(u,v), (\varphi, \psi)\rangle&:=a_1\int_{\mathbb{R}^3}\nabla u\cdot\nabla\varphi dx+a_2\int_{\mathbb{R}^3}\nabla v\cdot\nabla\psi 
dx+\int_{\mathbb{R}^3}V_1u\varphi dx+\int_{\mathbb{R}^3}V_2v\psi dx \\ 
      &\ \ \ \ +b_1\left(\int_{\mathbb{R}^3}|\nabla u|^2dx\right)\int_{\mathbb{R}^3}\nabla u\cdot\nabla\varphi dx+b_2\left(\int_{\mathbb{R}^3}|\nabla v|^2dx\right)\int_{\mathbb{R}^3}\nabla v\cdot\nabla\psi dx \\
      &\ \ \ -\mu\int_{\mathbb{R}^3}(I_{\alpha}*|u|^p)|u|^{p-2}\varphi dx-\nu\int_{\mathbb{R}^3}(I_{\alpha}*|v|^q)|v|^{q-2}\psi dx \\
      &\ \ \ -\lambda\int_{\mathbb{R}^3}u\psi dx-\lambda\int_{\mathbb{R}^3}v\varphi dx=0.
\end{align*}
  for any $(\varphi, \psi) \in C_0^{\infty}(\mathbb{R}^3) \times C_0^{\infty}(\mathbb{R}^3)$.
    
    We say that a weak solution $(u^*, v^*) \in H$ to \eqref{eq:NKC} is a \emph{ground state solution} if $I(u^*, v^*) \le I(u, v)$ for any other weak solution $(u, v) \in H \setminus \{(0, 0)\}$ to \eqref{eq:NKC}.

Main theorems of this paper are as follows. 

\begin{Th}\label{Th:non-critical}
Assume that \ref{assumption:potential} holds, and suppose that $\frac{3+\alpha}{3} < p, q < 3 + \alpha$.  
Then for any $\mu$, $\nu>0$  system \eqref{eq:NKC} admits a ground state solution.
\end{Th}

\begin{Th}\label{Th:critical} 
Assume that \ref{assumption:potential} holds.
\begin{enumerate}
\item[\textup{(1)}] Suppose that $\frac{3+\alpha}{3} < p < 3 + \alpha$ and $q = 3 + \alpha$. Then, for any fixed $\nu > 0$, there exists $\mu_0 > 0$ such that if $\mu \ge \mu_0$, system \eqref{eq:NKC} admits a ground state solution.

\item[\textup{(2)}] Suppose that $\frac{3+\alpha}{3} < q < 3 + \alpha$ and $p = \frac{3+\alpha}{3}$. Then, for any fixed $\mu > 0$, there exists $\nu_0 > 0$ such that if $\nu \ge \nu_0$, system \eqref{eq:NKC} admits a ground state solution.
\end{enumerate}
\end{Th}

Finally, we present a nonexistence result for the case where $p$, $q$ are same critical exponents. 

\begin{Th}\label{Th:nonexistence}
Assume that \ref{assumption:potential} holds, and suppose that $p = q = 3 + \alpha$ or $p = q = \frac{3+\alpha}{3}$.  
Then system \eqref{eq:NKC} has no nontrivial solution in $H$.
\end{Th}

For the case $p = 3 + \alpha$, $q = \frac{3 + \alpha}{3}$, the existence or nonexistence of a ground state solution remains an open problem.

  Throughout the paper we use the following notations:
    \begin{itemize}
        \item $C, C_1, C_2, \ldots$ denote positive constants possibly different.
        \item $B_R(x_0)$ denotes the open ball centered at $x_0 \in \mathbb{R}^3$ and radius $R > 0$.
        \item For $N\in\mathbb{N}$, a domain $\Omega\subset\mathbb{R}^N$ and $1\le s<\infty$ $L^s(\mathbb{R}^N)$ denotes the Lebesgue space with the norm $\|w\|_{L^s(\Omega)} = (\int_{\Omega}|w|^s\,dx)^{\frac{1}{s}}$. When $\Omega=\mathbb{R}^N$ we  use $\|w\|_{s}$ to express $\|w\|_{L^s(\mathbb{R}^N)}$.
        \item For any $w \in H^1(\mathbb{R}^3)$, we define $w^t(x) \coloneqq tw(t^{-2}x)$ for $t > 0$.
        \item We denote by $\mathcal{S}_N$ the best constant of the embedding $D^{1, 2}(\mathbb{R}^N) \hookrightarrow L^{2^*}(\mathbb{R}^N)$:
        \begin{align}
            \mathcal{S}_N\left(\int_{\mathbb{R}^N} |w|^{2^*}\,dx\right)^{\frac{2}{2^*}} \le \int_{\mathbb{R}^N} |\nabla w|^2\,dx,     \label{eq:Sobolev}
        \end{align}
        where $D^{1, 2}(\mathbb{R}^N) \coloneqq \{w \in L^{2^*}(\mathbb{R}^N) \mid |\nabla w| \in L^2(\mathbb{R}^N)\}$ and $2^*=2N/(N-2)$ for $N\ge 3$, that is,
        \begin{align}\label{eq:Sobolev-best}
            \mathcal{S}_N = \inf_{w \in D^{1, 2}(\mathbb{R}^N)} \frac{\|\nabla w\|_2^2}{\|w\|_{2^*}^2}.
        \end{align}
        \item When we treat the case where $p=(3+\alpha)/3$ or $q=3+\alpha$ the following two inequalities, which are the special case of the Hardy--Littlewood--Sobolev inequality, play a crucial role. We denote by $\mathcal{S}^*$ the best constant of the inequality 
        \begin{align}\label{eq:upper_critical_constant} 
            \mathcal{S}^*\left(\int_{\mathbb{R}^3} (I_{\alpha}*|w|^{3+\alpha})|w|^{3+\alpha}\,dx\right)^{\frac{1}{3+\alpha}} \le \int_{\mathbb{R}^3} |\nabla w|^2\,dx    \ \ \ \mbox{for}\ \ \ w\in D^{1,2}(\mathbb{R}^3) 
        \end{align}
        where $D^{1, 2}(\mathbb{R}^3) \coloneqq \{w \in L^6(\mathbb{R}^3) \mid |\nabla w| \in L^2(\mathbb{R}^3)\}$, that is,
        \begin{align*}
            \mathcal{S}^* = \inf_{w \in D^{1, 2}(\mathbb{R}^3)} \frac{\|\nabla w\|_2^2}{\displaystyle\left(\int_{\mathbb{R}^3}(I_{\alpha}*|w|^{3+\alpha})|w|^{3+\alpha}dx\right)^{\frac{3}{3+\alpha}}}.
        \end{align*}
        (see Moroz and Shaftingen \cite{Moroz-Schaftingen-3} and Seok \cite{Seok}) We denote by $\mathcal{S}_*$ the best constant of the inequality 
        \begin{align}\label{eq:lower_critical_constant} 
            \mathcal{S}_*\left(\int_{\mathbb{R}^3} (I_{\alpha}*|w|^{\frac{3+\alpha}{3}})|w|^{\frac{3+\alpha}{3}}\,dx\right)^{\frac{3}{3+\alpha}} \le \int_{\mathbb{R}^3} |w|^2\,dx    \ \ \ \mbox{for}\ \ \ w\in L^2(\mathbb{R}^3), 
        \end{align}
        that is,
        \begin{align*}
            \mathcal{S}_* = \inf_{w \in L^{2}(\mathbb{R}^3)} \frac{\|w\|_2^2}{\displaystyle\left(\int_{\mathbb{R}^3}(I_{\alpha}*|w|^{\frac{3+\alpha}{3}})|w|^{\frac{3+\alpha}{3}}dx\right)^{\frac{3}{3+\alpha}}}.
        \end{align*}
(see Moroz and Shaftingen \cite{Moroz-Schaftingen-3} and  Seok \cite{Seok}).
    \end{itemize}

This paper is organized as follows. In Section~2, we present some preliminary results.  
In Section~3, we establish the regularity of weak solutions and prove the corresponding Pohozaev identity by providing a detailed and self-contained proof.   

In Sections~4--6, we study the existence of ground state solutions via a minimization problem on a suitably defined Nehari--Pohozaev manifold.  
In Section~4, we introduce the Nehari--Pohozaev manifold and investigate its basic properties.  
Section~5 is devoted to the proof of Theorem~\ref{Th:non-critical}, dealing with the non-critical case.  
Finally, in Section~6, we prove Theorem~\ref{Th:critical}, the upper half and lower half critical cases. Proof of Theorem \ref{Th:nonexistence} will be given in Section~7.

\section{Preliminary Results}

To deal with the Choquard type nonlocal terms, the following Hardy--Littlewood--Sobolev inequality will be frequently used.
\begin{Lem}[Hardy--Littlewood--Sobolev inequality {\cite[Theorem 4.3]{Lieb-Loss}}]\label{H-L-S-ineq}
Let $r$, $s>1$ and $N\in\mathbb{N}$ be numbers which satisfy
\begin{align*}
\frac{1}{r}+\frac{1}{s}=1+\frac{\alpha}{N}.
\end{align*}
There exists a positive constant $C_{\mathrm{HLS}}(N,\alpha, r)$ such that
\begin{align*}
\left|\int_{\mathbb{R}^N}\int_{\mathbb{R}^N}\frac{f(x)h(y)}{|x-y|^{N-\alpha}}dxdy\right|\le C_{\mathrm{HLS}}(N,\alpha, r)\|f\|_r\|h\|_s.
\end{align*}
for all $f\in L^r(\mathbb{R}^N)$ and $h\in L^s(\mathbb{R}^N)$.
\end{Lem}
We note that since the embedding $H^1(\mathbb{R}^N)\hookrightarrow L^s(\mathbb{R}^N)$ $s\in (2,2N/(N-2)]$ is continuous, 
if $w\in H^1(\mathbb{R}^N)$ and $\frac{N+\alpha}{N}\le p\le \frac{N+\alpha}{N-2}$ then by using the Hardy--Littlewood--Sobolev inequality with $r=s=2N/(N+\alpha)$ we have
\begin{align}\label{eq:convolution-term-est}
\begin{split}
\int_{\mathbb{R}^3}(I_{\alpha}*|w|^p)|w|^pdx&\le C_{\mathrm{HLS}}(N,\alpha, r)\||w|^p\|_r\||w|^p\|_r \\
&=C_{\mathrm{HLS}}(N,\alpha, r)\left(\int_{\mathbb{R}^3}|w|^{\frac{2Np}{N+\alpha}}\right)^{\frac{N+\alpha}{N}} \\
&\le C_{\mathrm{HLS}}(N,\alpha, \alpha)C_{\mathrm{S}}(N,\alpha,p)^{2p}\|w\|_{H^1(\mathbb{R}^N)}^{2p},
\end{split}
\end{align}
where $C_{\mathrm{S}}(N,\alpha,p)$ is a constant determined by the embedding $H^1(\mathbb{R}^N)\hookrightarrow L^{2Np/(N+\alpha)}(\mathbb{R}^N)$.

Let us recall the nonlocal version of the Brezis--Lieb lemma. 
\begin{Lem}\label{lem:nonlocal-B-L}
Let $N\in\mathbb{N}$, $\alpha\in (0,N)$, $p\in (\frac{N+\alpha}{2N}, \infty)$ and $\{w_n\}$ be a bounded sequence in $L^{\frac{2Np}{N+\alpha}}(\mathbb{R}^N)$. If $w_n\to w$ almost everywhere on $\mathbb{R}^N$, then
\begin{align*}
\lim_{n\to\infty}\left\{\int_{\mathbb{R}^N}(I_{\alpha}*|w_n|^p)|w_n|^pdx-\int_{\mathbb{R}^N}(I_{\alpha}*|w_n-w|^p)|w_n-w|^pdx\right\}=\int_{\mathbb{R}^N}(I_{\alpha}*|w|^p)|w|^pdx. 
\end{align*}
\end{Lem}
For reader's convenience we give the proof of Lemma \ref{lem:nonlocal-B-L}.  We need a variant of the original Brezis--Lieb Lemma. 
\begin{Lem}\label{lem:B-L-2}
Let $\Omega\subset\mathbb{R}^N$ be a domain, $r\in [1,\infty)$ and $\{w_n\}$ be a bounded sequence in $L^r(\Omega)$. If $\lim_{n\to\infty}w_n(x)=w(x)$ for almost all $x\in\Omega$, then for every $q\in [1,r]$
\begin{align*}
\lim_{n\to\infty}\int_{\Omega}\left||w_n|^q-|w_n-w|^q-|w|^q\right|^{\frac{r}{q}}dx=0.
\end{align*}
\end{Lem}
\begin{proof}
Although this lemma can be proved by adapting the argument of the proof of Theorem~4.2.7 in~\cite{Willem-functional-analysis}, we provide the proof here for the reader’s convenience.
 Since $\{w_n\}$ is bounded in $L^r(\Omega)$ there exists $M>0$ such that $\|w_n\|_r\le M$ $(n=1,2,\cdots)$. By Fatou's lemma, we see $\|w\|_r\le M$. By a simple consideration we can show that for any $\varepsilon>0$ there exists $c_{\varepsilon}>0$ such that
\begin{align*}
||a+b|^q-|a|^q-|b|^q|\le \varepsilon |a|^q+c_{\varepsilon}|b|^q.
\end{align*}
Hence for $q\in [1,r]$ we have
\begin{align*}
||a+b|^q-|a|^q-|b|^q|^{\frac{r}{q}}\le 2^{\frac{r}{q}}(\varepsilon^{\frac{r}{q}}|a|^r+c_{\varepsilon}^{\frac{r}{q}}|b|^r).
\end{align*}
Thus we obtain
\begin{align*}
||w_n|^q-|w_n-w|^q-|w|^q|^{\frac{r}{q}}\le 2^{\frac{r}{q}}(\varepsilon^{\frac{r}{q}}|w_n-w|^r+c_{\varepsilon}^{\frac{r}{q}}|w|^r).
\end{align*}
and $||w_n|^q-|w_n-w|^q-|w|^q|^{\frac{r}{q}}\in L^1(\mathbb{R})$. 

Since 
\begin{align*}
 2^{\frac{r}{q}}(\varepsilon^{\frac{r}{q}}|w_n-w|^r+c_{\varepsilon}^{\frac{r}{q}}|w|^r)-||w_n|^q-|w_n-w|^q-|w|^q|^{\frac{r}{q}}\to (2c_{\varepsilon})^{\frac{r}{q}}|w|^r\ \ \mbox{a.e.}\ x\in\Omega, 
\end{align*}
by Fatou's lemma and the fact $\|w_n-w\|_{L^r(\Omega)}\le 2M$ it follows that
\begin{align*}
   &\ (2c_{\varepsilon})^{\frac{r}{q}}\int_{\Omega}|w|^rdx \\
\le&\ \liminf_{n\to\infty}\int_{\Omega}\left\{2^{\frac{r}{q}}(\varepsilon^{\frac{r}{q}}|w_n-w|^r+c_{\varepsilon}^{\frac{r}{q}}|w|^r)-||w_n|^q-|w_n-w|^q-|w|^q|^{\frac{r}{q}}\right\}dx \\
\le &\ \liminf_{n\to\infty}\left\{(2\varepsilon)^{\frac{r}{q}}\|w_n-w\|_{L^r(\Omega)}^r+(2c_{\varepsilon})^{\frac{r}{q}}\int_{\Omega}|w|^rdx\right\}-\limsup_{n\to\infty}\int_{\Omega}||w_n|-|w_n-w|^q-|w|^q|^{\frac{r}{q}}dx \\
 \le &\ \liminf_{n\to\infty}\left\{(2\varepsilon)^{\frac{r}{q}}(2M)^r+(2c_{\varepsilon})^{\frac{r}{q}}\int_{\Omega}|w|^rdx\right\}-\limsup_{n\to\infty}\int_{\Omega}||w_n|-|w_n-w|^q-|w|^q|^{\frac{r}{q}}dx \\
=&\ (2\varepsilon)^{\frac{r}{q}}(2M)^r+\left\{(2c_{\varepsilon})^{\frac{r}{q}}\int_{\Omega}|w|^rdx\right\}-\limsup_{n\to\infty}\int_{\Omega}||w_n|-|w_n-w|^q-|w|^q|^{\frac{r}{q}}dx 
\end{align*}
and then
\begin{align*}
\limsup_{n\to\infty}\int_{\Omega}||w_n|-|w_n-w|^q-|w|^q|^{\frac{r}{q}}dx\le (2\varepsilon)^{\frac{r}{q}}(2M)^r. 
\end{align*}
Since $\varepsilon > 0$ was chosen arbitrarily, we conclude that
\begin{align*}
\lim_{n \to \infty} \int_{\Omega} \left|\, |w_n|^q - |w_n - w|^q - |w|^q \, \right|^{\frac{r}{q}} \, dx = 0.
\end{align*}
\end{proof}

Next lemma states that any bounded sequence in 
$L^r(\Omega)$ that converges almost everywhere also converges weakly.

\begin{Lem}[{\cite[Proposition 5.4.7]{Willem-functional-analysis}}]\label{lemma:ae_to_weak}
Let $\Omega\subset\mathbb{R}^N$ be a domain and $1<r<\infty$. Suppose that $\{w_n\}$ is a bounded sequence in $L^r(\Omega)$. If $w_n(x)$ converges $w(x)$ almost every $x\in\Omega$, then $w_n\rightharpoonup w$ weakly in $L^r(\Omega)$.
\end{Lem}

We now give the proof of Lemma \ref{lem:nonlocal-B-L}. 

\begin{proof}[Proof of Lemma \ref{lem:nonlocal-B-L}]
We first note that for every $n\in\mathbb{N}$
\begin{align*}
 &\ \int_{\mathbb{R}^N}(I_{\alpha}*|w_n|^p)|w_n|^pdx-\int_{\mathbb{R}^N}(I_{\alpha}*|w_n-w|^p)|w_n-w|^pdx-\int_{\mathbb{R}^3}(I_{\alpha}*|w|^p)|w|^pdx \\
=&\ \int_{\mathbb{R}^N}(I_{\alpha}*(|w_n|^p-|w_n-w|^p))(|w_n|^p-|w_n-w|^p-|w|^p)dx \\
  &\ \ \ \ \ \ \ \ \ \ \ \ + \int_{\mathbb{R}^N}(I_{\alpha}*(|w_n|^p-|w_n-w|^p-|w|^p))(|w_n|^p-|w_n-w|^p)dx  \\
 &\ +2\int_{\mathbb{R}^N}(I_{\alpha}*(|w_n|^p-|w_n-w|^p))|w_n-w|^pdx \\
 =&:I_1+I_2+2I_3
\end{align*}
By choosing $r=\frac{2Np}{N+\alpha}$, $q=p$ in Lemma \ref{lem:B-L-2} we see $|w_n|^p-|w_n-w|^p\to |w|^p$ in $L^{\frac{2N}{N+\alpha}}(\mathbb{R}^N)$. By the Hardy--Littlewood--Sobolev inequality we have
\begin{align*}
|I_1|, |I_2|&\le C_{\mathrm{HLS}}\left(N,\alpha,{\textstyle{\frac{2N}{N+\alpha}}}\right)\left(\|w_n\|_{\frac{2Np}{N+\alpha}}^{\frac{N+\alpha}{2N}}+\|w_n-w\|_{\frac{2N}{N+\alpha}}^{\frac{N+\alpha}{2N}}\right)\left\||w_n|^p-|w_n-w|^p-|w|^p\right\|_{\frac{2N}{N+\alpha}} \\
 &\to 0.
\end{align*}
We next note that, by Lemma \ref{lemma:ae_to_weak}, $|w_n-w|^p\rightharpoonup 0$ in $L^{\frac{2N}{N+\alpha}}(\mathbb{R}^N)$ as $n\to\infty$. Since $I_{\alpha}*|w|^p\in L^{\frac{2N}{N-\alpha}}(\mathbb{R}^N)$ we have
\begin{align*}
|I_3|\le \left|\int_{\mathbb{R}^3}(I_{\alpha}*(|w_n|^p-|w_n-w|^p-|w|^p)\|w_n-w|^pdx\right|+\left|\int_{\mathbb{R}^3}(I_{\alpha}*|w|^p)|w_n-w|^pdx\right|\to 0.
\end{align*}
The proof is now complete.
\end{proof}

Finally, let us recall the following vanishing lemma by P. L. Lions\cite{P.L.Lions}. 
    \begin{Lem}[{\cite[Lemma I.1]{P.L.Lions}, \cite[Lemma 1.21]{Willem}}]\label{Lions-th}
    Suppose that $\{w_n\}$ is a bounded sequence in $H^1(\mathbb{R}^3)$. If $\{w_n\}$ satisfies
    \begin{align*}
\lim_{n\to\infty}\sup_{y\in\mathbb{R}^3}\int_{B_R(y)}|w_n|^2dx=0
    \end{align*}
    for some $R>0$, then $w_n\to 0$ as $n\to\infty$ in $L^r(\mathbb{R}^3)$ for any $r\in (2,6)$.
    \end{Lem}

\section{Regularity and Pohozaev identity}

    Since our approach is a minimization technique over the Nehari--Pohozaev manifold, we need the Pohozaev identity.
    \begin{Lem}[Pohozaev identity]  \label{Lem:Pohozaev}
        If $(u, v) \in H$ is a weak solution to \eqref{eq:NKC}, then $(u, v)$ satisfies the following Pohozaev identity:
        \begin{align}\label{eq:Pohozaev-id}
            P(u, v) &\coloneqq \frac{1}{2}(a_1\|\nabla u\|_2^2 + a_2\|\nabla v\|_2^2) + \frac{3}{2}\left(V_1\|u\|_2^2 + V_2\|v\|_2^2\right)+\frac{1}{2}(b_1\|\nabla u\|_2^4 + b_2\|\nabla v\|_2^4) \\
            &\quad- \frac{3+\alpha}{2p}\mu\int_{\mathbb{R}^3}(I_{\alpha}*|u|^p)|u|^pdx - \frac{3+\alpha}{2q}\nu\int_{\mathbb{R}^3}(I_{\alpha}*|v|^q)|v|^qdx\\ 
            &\quad -3 \lambda \int_{\mathbb{R}^3} uv\,dx = 0.
        \end{align}
    \end{Lem}

\subsection{Regularity of Solutions to a Linearly Coupled Choquard System}
To get the above Pohozaev identity, the solution is required to be in $H^2_{\rm{loc}}(\mathbb{R}^3)$. To this end we establish the regularity result 
for the following linearly coupled Choqurd system on $\mathbb{R}^N$:
\begin{align}\label{eq:Choquard-system}
\begin{cases}
-d_1\Delta u+V_1u=\mu(I_{\alpha}*|u|^p)|u|^{p-2}u+\lambda u\ \ x\in\mathbb{R}^N, \\
-d_2\Delta v+V_2v=\nu(I_{\alpha}*|v|^q)|v|^{q-2}v+\lambda v\ \ x\in\mathbb{R}^N,\\
u,v\in H^1(\mathbb{R}^N),
\end{cases}
\end{align}
where $N\in\mathbb{N}$, $d_1$, $d_2$, $V_1$, $V_2$, $\mu$, $\nu$ are positive constants and $\lambda\in\mathbb{R}$ is a constant and the Riesz potential $I_{\alpha}$ here is given by 
\begin{align*}
I_{\alpha}(x)=\frac{A_{\alpha}}{|x|^{N-\alpha}}\ \ \mbox{where}\ \ A_{\alpha}=\frac{\Gamma\left(\frac{N-\alpha}{2}\right)}{\Gamma\left(\frac{\alpha}{2}\right)\pi^{\frac{\alpha}{2}}2^{\alpha}}.
\end{align*}
In this subsection, we assume that $p,q\in \left[1+\frac{\alpha}{N}, \frac{N+\alpha}{N-2}\right]$.

For the single equation \eqref{eq:Choquard},
Moroz and Van Schaftingen proved that if $u \in H^1(\mathbb{R}^N)$ is a weak solution to \eqref{eq:Choquard}, then $u \in W_{\mathrm{loc}}^{2,r}(\mathbb{R}^N)$ holds for any $r \ge 1$. This result was established in \cite{Moroz-Schaftingen} for the case $1 + \frac{\alpha}{N} < p < \frac{N + \alpha}{N - 2}$, and extended in \cite{Moroz-Schaftingen-3} to the critical cases $p = 1 + \frac{\alpha}{N}$ and $p = \frac{N + \alpha}{N - 2}$.  Although Yang, Albuquerque, Silva, and Silva mention in \cite{Y-A-S-S} that the Pohozaev identity corresponding to the system \eqref{eq:Choquard-system} can be derived by adapting arguments in \cite[Theorem 3]{Moroz-Schaftingen-3} and \cite[Lemma 6.1]{J.M.do-O}, the proof of the regularity required for the validity of the identity is not provided. In this paper, for the reader’s convenience, we present a detailed proof of the regularity result for weak solutions to \eqref{eq:Choquard-system}, following the approaches of Moroz and Van Schaftingen \cite{Moroz-Schaftingen-3}, Li and Ma \cite{Li-Ma}, and Li, Li, and Tang~\cite{Li-Li-Tang}.

We begin by presenting the following improved integrability result for a linearly coupled linear elliptic system involving nonlocal terms.

\begin{Prop}\label{prop:reg} Let $N\ge 3$ and $\alpha\in (0,N)$. If $H_1$, $H_2$, $K_1$, $K_2\in L^{\frac{2N}{\alpha+2}}(\mathbb{R}^N)+L^{\frac{2N}{\alpha}}(\mathbb{R}^N)$, $\lambda\in\mathbb{R}$ and $(u,v)\in H^1(\mathbb{R}^N)\times H^1(\mathbb{R}^N)$ solves
\begin{align}\label{eq:Choquard-system-aux}
\begin{cases}
-d_1\Delta u+V_1 u=(I_{\alpha}*(H_1u))K_1+\lambda v, \ \ \ x\in\mathbb{R}^N, \\
-d_2\Delta v+V_2 v=(I_{\alpha}*(H_2v))K_2+\lambda u,\ \ \ x\in\mathbb{R}^N,
\end{cases}
\end{align}
where $d_1$, $d_2$, $V_1$, $V_2$ are positive constants and $\lambda$ is a real number, then $u\in L^r(\mathbb{R}^N)$ for $r\in [2, \frac{N}{\alpha} \frac{2N}{N-2})$. Moreover there exists a positive constant $C(r, V_1, V_2, \lambda)$ independent of $u$ such that
\begin{align*}
\|u\|_r+\|v\|_r\le C(r, V_1, V_2, \lambda)(\|u\|_2+\|v\|_2). 
\end{align*}
\end{Prop}
To prove Proposition \ref{prop:reg}, we recall the following lemma from \cite{Moroz-Schaftingen-3}. 

\begin{Lem}[{\cite[Lemma 3.2]{Moroz-Schaftingen-3}}]\label{lem:reg-lemma}
Let $N\ge 2$, $\alpha\in (0,N)$ and $\theta\in (0,2)$. If $H$, $K\in L^{\frac{2N}{N+\alpha}}(\mathbb{R}^N)+L^{\frac{2N}{\alpha}}(\mathbb{R}^N)$ and $\frac{\alpha}{N}<\theta<2-\frac{\alpha}{N}$, then for every $\varepsilon>0$, 
there exists a positive constant $C_{\varepsilon,\theta}>0$ such that for every $u\in H^1(\mathbb{R}^N)$,
\begin{align*}
\int_{\mathbb{R}^N}(I_{\alpha}*(H|u|^{\theta}))K|u|^{2-\theta}dx\le \varepsilon^2\int_{\mathbb{R}^N}|\nabla u|^2dx+C_{\varepsilon,\theta}\int_{\mathbb{R}^N}|u|^2dx.
\end{align*}
\end{Lem}

\begin{proof}[Proof of Proposition \ref{prop:reg}]
For the simplicity we assume that $d_1=d_2=1$. We follow the proof of \cite[Lemma 2.4]{Li-Ma}, \cite[Proposition 3.1]{Moroz-Schaftingen-3} and \cite[Proposition 2.2]{Li-Li-Tang}. Set $H_j=H_{1j}+H_{2j}$, $K_j=K_{1j}+K_{2j}$ $(j=1,2)$, where $H_{1j}$, $K_{1j}\in L^{\frac{2N}{\alpha}}(\mathbb{R}^N)$, $H_{2j}$, $K_{2j}\in L^{\frac{2N}{2+\alpha}}(\mathbb{R}^N)$ $(j=1,2)$.  By Lemma \ref{lem:reg-lemma} with $\theta=1$, there exists a constant $\kappa>\max\{V_1,V_2\}$ such that for every $w\in H^1(\mathbb{R}^N)$,
\begin{align}\label{eq:reg-01}
\int_{\mathbb{R}^N}(I_{\alpha}*((|H_{1j}|+|H_{2j}|)|w|))((|K_{1j}|+|K_{2j}|)|w|)dx\le\frac{1}{4}\int_{\mathbb{R}^N}|\nabla w|^2dx+\frac{\kappa}{4}\int_{\mathbb{R}^N}|w|^2dx.
\end{align}
For any function $M(x)$ and each $k\in\mathbb{N}$ we define $M^{(k)}$ by
\begin{align*}
M^{(k)}(x)=
\begin{cases}
k &\mbox{if}\ M(x)>k, \\
M(x) &\mbox{if}\ |M(x)|\le k \\
-k  &\mbox{if}\ M(x)<-k.
\end{cases}
\end{align*}
Let $\overline{H}_{j,k}:=H_{1j}+H_{2j}^{(k)}$ and $\overline{K}_{j,k}:=K_{1j}+K_{2j}^{(k)}$. Then $\{\overline{H}_{j,k}\}$ and $\{\overline{K}_{j,k}\}$ satisfy
\begin{align*}
&\overline{H}_{j,k}, \overline{K}_{j,k}\in L^{\frac{2N}{\alpha}}(\mathbb{R}^N), \\
&|\overline{H}_{j,k}|\le |H_{1j}|+|H_{2j}|,\ |\overline{K}_{j,k}|\le |K_{1j}|+|K_{2j}|, \\
&\overline{H}_{j,k}\to H_j,\ \overline{K}_{j,k}\to K_j\ \ \mbox{a.e.\ on}\ \ \mathbb{R}^N.
\end{align*}
For each $j=1,2$ and $k\in\mathbb{N}$ we define a bilinear form $a_j^{(k)}:H^1(\mathbb{R}^N)\times H^1(\mathbb{R}^N)\to\mathbb{R}$ by
\begin{align*}
a_j^{(k)}(\varphi,\psi)=\int_{\mathbb{R}^N}\nabla\varphi\cdot\nabla\psi+\kappa\varphi\psi dx-\int_{\mathbb{R}^N}(I_{\alpha}*(\overline{H}_{j,k}\varphi))\overline{K}_{j,k}\psi.
\end{align*}

By \eqref{eq:reg-01} we see that $a_j^{(k)}$ is bounded and 
\begin{align}\label{eq:reg-02}
a_j^{(k)}(\varphi, \varphi)\ge\frac{1}{2}\int_{\mathbb{R}^N}|\nabla\varphi|^2dx+\frac{\kappa}{2}\int_{\mathbb{R}^N}|\varphi|^2dx
\end{align}
for any $\varphi\in H^1(\mathbb{R}^N)$, $j=1,2$ and $k\in\mathbb{N}$, that is $a_j^{(k)}$ is coercive. 

The Lax--Milgram theorem implies that there exists a unique solution $(u_k, v_k)\in H^1(\mathbb{R}^N)\times H^1(\mathbb{R}^N)$ to the following problem
\begin{align}\label{eq:reg-00}
\begin{split}
&-\Delta u_k+\kappa u_k=(I_{\alpha}*(\overline{H}_{1,k}u_k))\overline{K}_{1,k}+\lambda v+(\kappa-V_1)u\ \ \mbox{in}\ \ \mathbb{R}^N, \\
&-\Delta v_k+\kappa v_k=(I_{\alpha}*(\overline{H}_{2,k}v_k))\overline{K}_{2,k}+\lambda u+(\kappa-V_2)v\ \ \mbox{in}\ \ \mathbb{R}^N,
\end{split}
\end{align}
where $(u, v)\in H^1(\mathbb{R}^N)\times H^1(\mathbb{R}^N)$ is the solution to \eqref{eq:Choquard-system-aux}.  

\noindent
\textbf{Claim:} $u_k\rightharpoonup u$, $v_k\rightharpoonup v$ in $H^1(\mathbb{R}^N)$. 

From \eqref{eq:reg-02} and \eqref{eq:reg-00} we obtain
\begin{align*}
  &\ \frac{1}{2}\int_{\mathbb{R}^N}|\nabla u_k|^2dx+\frac{\kappa}{2}\int_{\mathbb{R}^N}|u_k|^2dx\\
 \le&\ a_1^{(k)}(u_k, u_k) \\
 =&\lambda\int_{\mathbb{R}^N}vu_kdx+(\kappa-V_1)\int_{\mathbb{R}^N}uu_k dx \\
 \le&\ \frac{\kappa}{8}\int_{\mathbb{R}^N}|u_k|^2dx+\frac{2\lambda^2}{\kappa}\int_{\mathbb{R}^N}|v|^2dx+2\frac{(\kappa-V_1)^2}{\kappa}\int_{\mathbb{R}^N}|u|^2dx+\frac{\kappa}{8}\int_{\mathbb{R}^N}|u_k|^2dx
\end{align*}
and then
\begin{align}\label{eq:reg-00-01}
\frac{\min\{2,\kappa\}}{4}\|u_k\|_{H^1(\mathbb{R}^N)}^2\le \frac{2\lambda^2}{\kappa}\int_{\mathbb{R}^N}|v|^2dx+\frac{2(\kappa-V_1)^2}{\kappa}\int_{\mathbb{R}^N}|u|^2dx.
\end{align}
Hence $\{u_k\}$ is bounded in $H^1(\mathbb{R}^N)$. Similary we can show that $\{v_k\}$ is bounded in $H^1(\mathbb{R}^N)$. 

Along a subsequence we may assume that there exists $\overline{u}$, $\overline{v}\in H^1(\mathbb{R}^N)$ such that $u_k\to \overline{u}$, $v_k\to \overline{v}$ in $L^s_{\mathrm{loc}}(\mathbb{R}^N)$ for any $s\in [1,2N/(N-2))$ and $u_k\to \overline{u}$  and $v_k\to \overline{v}$ almost everywhere in $\mathbb{R}^N$. Since $|\overline{K}_{j,k}|\le |K_{1j}|+|K_{2j}|$ and $|\overline{H}_{j,k}|\le |H_{1j}|+|H_{2j}|$, $\{\overline{K}_{j,k}\}_k$ and $\{\overline{H}_{j,k}\}_k$ are bounded in $L^{\frac{2N}{2+\alpha}}(\mathbb{R}^N)+L^{\frac{2N}{\alpha}}(\mathbb{R}^N)$. By Lemma \ref{lemma:ae_to_weak} it holds that $\overline{H}_{j,k}u_k\rightharpoonup H_j\overline{u}$ in $L^{\frac{2N}{N+\alpha}}(\mathbb{R}^N)$.

By the Lebesgue dominated convergence theorem we have $\overline{K}_{j,k}\varphi\to K_j\varphi$ strongly in $L^{\frac{2N}{N+\alpha}}(\mathbb{R}^N)$ for any $\varphi\in C_0^{\infty}(\mathbb{R}^N)$.  

 Let us define the functionals
\begin{align*}
F_j(w)=\int_{\mathbb{R}^N}(I_{\alpha}*w)K_j\varphi dx\ \ (j=1,2)
\end{align*}
for $w\in L^{\frac{2N}{N+\alpha}}(\mathbb{R}^N)$. These are linear functionals on $L^{\frac{2N}{N+\alpha}}(\mathbb{R}^N)$. Moreover, the Hardy--Littlewood--Sobolev inequality implies that
\begin{align*}
|F_j(w)|\le C\|w\|_{\frac{2N}{N+\alpha}}\|K_j\varphi\|_{\frac{2N}{N+\alpha}}.
\end{align*}
This means that $F_j$ is a bounded linear functional on $L^{\frac{2N}{N+\alpha}}(\mathbb{R}^N)$. Hence $F_j(\overline{H}_{j,k}u_k)\to F_j(H_j\overline{u})$, and then
\begin{align}\label{eq:reg-03}
\int_{\mathbb{R}^N}(I_{\alpha}*(\overline{H}_{j,k}u_k))\overline{K}_{j,k}\varphi dx\to\int_{\mathbb{R}^N}(I_{\alpha}*(H_j\overline{u}))K_j\varphi dx.
\end{align}
By \eqref{eq:reg-00} and \eqref{eq:reg-03} we see that $\overline{u}$ satisfies
\begin{align*}
-\Delta\overline{u}+\kappa\overline{u}=(I_{\alpha}*(H_1\overline{u}))K_1+\lambda v+(\kappa-V_1)u\ \ \mbox{in}\ \ \mathbb{R}^N, 
\end{align*}
Since the weak solution to
\begin{align*}
-\Delta w+\kappa w=(I_{\alpha}*(H_1w))K_1+\lambda v+(\kappa-V_1)u\ \ \mbox{in}\ \ \mathbb{R}^N, 
\end{align*}
is unique we conclude that $\overline{u}=u$. Similarly we can conclude that $\overline{v}=v$. 

For $\eta>0$, we define the truncation $u_{k}^{(\eta)}$ of $u_k$ by
\begin{align*}
u_{k}^{(\eta)}(x)=
\begin{cases}
\eta, & u_{k}(x)>\eta, \\
u_k(x), & -\eta\le u_k(x)\le \eta, \\
-\eta & u_k(x)<-\eta.
\end{cases}
\end{align*}
By considering $G_{\mu}\in C^1(\mathbb{R})$ which satisfies $G_{\eta}(s)=|s|^{r-2}s$ for $|s|\le \eta+1$ and $G'\in L^{\infty}(\mathbb{R})$ we see that $G_{\eta}(u_k^{(\eta)})=|u_k^{(\eta)}|^{r-2}u_k^{(\eta)}\in H^1(\mathbb{R}^N)$ for any $r\ge 2$(see \cite[Proposition 9.5]{Brezis}). By taking it as a test function in the weak form of the first equation in \eqref{eq:reg-00} we have
\begin{align*}
&\int_{\mathbb{R}^N}\nabla u_k\cdot \nabla(|u_k^{(\eta)}|^{r-2}u_{k,\eta})dx+\kappa\int_{\mathbb{R}^N}u_k(|u_k^{(\eta)}|^{r-2}u_k^{(\eta)})dx \\
&\ \ \ \ \ \ \ \ \ =\int_{\mathbb{R}^N}(I_{\alpha}*(\overline{H}_{1,k}u_k))\overline{K}_{1,k}|u_k^{(\eta)}|^{r-2}u_k^{(\eta)}dx+\lambda\int_{\mathbb{R}^N}v|u_k^{(\eta)}|^{r-2}u_k^{(\eta)}dx \\
&\ \ \ \ \ \ \ \ \ \ \ \ \ \ \ \ +(\kappa-V_1)\int_{\mathbb{R}^N}u|u_k^{(\eta)}|^{r-2}u_k^{(\eta)}dx.
\end{align*}
Since $\nabla |u_k^{(\eta)}|^{r-2}u_k^{(\eta)}=(r-1)|u_k^{(\eta)}|^{r-2}\nabla u_k^{(\eta)}$, $\nabla u_k^{(\eta)}=0$ on 
\begin{align*}
A_{k,\eta}:=\{x\in\mathbb{R}^N:|u_k(x)|>\eta\}
\end{align*}
and the fact that $|u_k^{(\eta)}|^2\le u_ku_k^{(\eta)}$ on $\mathbb{R}^N$
we have
\begin{align*}
\int_{\mathbb{R}^N}\nabla u_k\cdot \nabla(|u_k^{(\eta)}|^{r-2}u_k^{(\eta)})dx&=(r-1)\int_{\mathbb{R}^N}|u_k^{(\eta)}|^{r-2}\nabla u_k\cdot\nabla u_k^{(\eta)}dx \\
 &=(r-1)\int_{\mathbb{R}^N}|u_k^{(\eta)}|^{r-2}|\nabla u_k^{(\eta)}|^2dx
\end{align*}
and
\begin{align*}
\int_{\mathbb{R}^N}|u_k^{(\eta)}|^rdx\le\int_{\mathbb{R}^N}|u_k^{(\eta)}|^{r-2}u_k^{(\eta)}u_kdx. 
\end{align*}

On the other hand we see $\nabla|u_k^{(\eta)}|^{\frac{r}{2}}=\frac{r}{2}|u_k^{(\eta)}|^{\frac{r-4}{2}}u_k^{(\eta)}\nabla u_k^{(\eta)}$ (for the case $r=2$ see \cite[Lemma 7.6 and Lemma 7.7]{Gilberg-Trudinger}). Hence we obtain
\begin{align}\label{eq:reg-04-00}
\begin{split}
  &\ \frac{4(r-1)}{r^2}\int_{\mathbb{R}^N}|\nabla|u_k^{(\eta)}|^{\frac{r}{2}}|^2dx+\kappa\int_{\mathbb{R}^N}(|u_k^{(\eta)}|^{\frac{r}{2}})^2dx\\
\le&\ \int_{\mathbb{R}^N}\nabla u_k\cdot \nabla(|u_k^{(\eta)}|^{r-2}u_k^{(\eta)})dx+\kappa\int_{\mathbb{R}^N}u_k(|u_k^{(\eta)}|^{r-2}u_k^{(\eta)})dx \\
=&\ \int_{\mathbb{R}^N}(I_{\alpha}*(\overline{H}_{1,k}u_k))\overline{K}_{1,k}|u_k^{(\eta)}|^{r-2}u_k^{(\eta)}dx+\lambda\int_{\mathbb{R}^N}v|u_k^{(\eta)}|^{r-2}u_k^{(\eta)}dx \\
 &\ \ \ \ \ \ \ \ \ +(\kappa-V_1)\int_{\mathbb{R}^N}u|u_k^{(\eta)}|^{r-2}u_k^{(\eta)}dx.
\end{split}
\end{align}
We next note that
\begin{align}\label{eq:reg-04-01}
\begin{split}
 &\int_{\mathbb{R}^N}(I_{\alpha}*(\overline{H}_{1,k}u_k))\overline{K}_{1,k}|u_k^{(\eta)}|^{r-2}u_k^{(\eta)}dx \\
=&\int_{\mathbb{R}^N}(I_{\alpha}*(\overline{H}_{1,k}u_{k}^{(\eta)}))\overline{K}_{1,k}|u_{k}^{(\eta)}|^{r-2}u_{k}^{(\eta)}dx \\
 &\ \ \ \ \ \ \ \ \ \ \ \ \ \ \ \ \ \ \ \ \ \ \ \ \ \ \ \ \ \ \ + \int_{\mathbb{R}^N}(I_{\alpha}*(\overline{H}_{1,k}(u_k-u_k^{(\eta)}))\overline{K}_{1,k}|u_k^{(\eta)}|^{r-2}u_k^{(\eta)}dx \\
=&\int_{\mathbb{R}^N}(I_{\alpha}*(\overline{H}_{1,k}u_k^{(\eta)}))\overline{K}_{1,k}|u_k^{(\eta)}|^{r-2}u_k^{(\eta)}dx \\
 &\ \ \ \ \ \ \ \ \ \ \ \ \ \ \ \ \ \ \ \ \ \ \ \ \ \ \ \ \ \ \ + \int_{\mathbb{R}^N}(I_{\alpha}*\overline{K}_{1,k}|u_k^{(\eta)}|^{r-2}u_k^{(\eta)})(\overline{H}_{1,k}(u_k-u_k^{(\eta)}))dx.
 \end{split}
\end{align}
If $r<\frac{2N}{\alpha}$, then  by Lemma \ref{lem:reg-lemma} with $\theta=\frac{2}{r}$, there exists $C_r>0$ such that
\begin{align}\label{eq:reg-04-02}
\begin{split}
&\int_{\mathbb{R}^N}(I_{\alpha}*(\overline{H}_{1,k}u_k^{(\eta)}))\overline{K}_{1,k}|u_k^{(\eta)}|^{r-2}u_k^{(\eta)}dx \\
&\ \ \ \ \le \int_{\mathbb{R}^N}(I_{\alpha}*(|\overline{H}_{1,k}||u_k^{(\eta)}|)|\overline{K}_{1,k}||u_k^{(\eta)}|^{r-1}dx \\
&\ \ \ \ = \int_{\mathbb{R}^N}\left(I_{\alpha}*\left(|\overline{H}_{1,k}|\left||u_k^{(\eta)}|^{\frac{r}{2}}\right|^{\frac{2}{r}}\right)\right)|\overline{K}_{1,k}|\left||u_k^{(\eta)}|^{\frac{r}{2}}\right|^{2-\frac{2}{r}}dx \\
&\ \ \ \ \le \frac{2(r-1)}{r^2}\int_{\mathbb{R}^N}|\nabla |u_k^{(\eta)}|^{\frac{r}{2}}|^2dx+C_r\int_{\mathbb{R}^N}||u_k^{(\eta)}|^{\frac{r}{2}}|^2dx \\
&\ \ \ \ \le \frac{2(r-1)}{r^2}\int_{\mathbb{R}^N}|\nabla |u_k^{(\eta)}|^{\frac{r}{2}}|^2dx+C_r\int_{\mathbb{R}^N}|u_k|^rdx.
\end{split}
\end{align}
We next estimate the second term of \eqref{eq:reg-04-01} by using the Hardy--Littlwood--Sobolev inequality. Set $\frac{1}{s}=\frac{\alpha}{2N}+1-\frac{1}{r}$ and $\frac{1}{t}=\frac{\alpha}{2N}+\frac{1}{r}$. Since $r\in [2,\frac{2N}{\alpha})$ we have $1<s\le\frac{2N}{N+\alpha}$ and $\frac{2N}{N+\alpha}\le t<\frac{N}{\alpha}$. Hence the Hardy--Littlewood--Sobolev inequality, the fact that $u_k-u_{k,\eta}=0$ on $\mathbb{R}^N\setminus A_{k,\eta}$ and $|u_k^{(\eta)}|\le |u_k|$, $|u_k-u_k^{(\eta)}|\le 2|u_k|$ on $\mathbb{R}^N$ imply that
\begin{align*}
&\int_{\mathbb{R}^N}(I_{\alpha}*\overline{K}_{1,k}|u_k^{(\eta)}|^{r-2}u_k^{(\eta)})(\overline{H}_{1,k}(u_k-u_k^{(\eta)}))dx. \\ 
&\ \ \ \le \int_{\mathbb{R}^N}(I_{\alpha}*|\overline{K}_{1,k}||u_k^{(\eta)}|^{r-1})|\overline{H}_{1,k}||u_k-u_k^{(\eta)}|dx \\
&\ \ \ \le 2C_{\mathrm{HLS}}\left(N,\alpha, {\textstyle\frac{2N}{r\alpha+2N}}\right)\left(\int_{\mathbb{R}^N}|\overline{K}_{1,k}|^s|u_{k}^{(\eta)}|^{s(r-1)}dx\right)^{\frac{1}{s}}\left(\int_{A_{\eta,k}}|\overline{H}_{1,k}|^t|u_k|^{t}dx\right)^{\frac{1}{t}}.
\end{align*}
Since $\frac{\alpha s}{2N}+\frac{r-1}{r}s=1$ and $\frac{\alpha}{2N}t+\frac{t}{r}=1$, by using the H\"{o}lder inequality, if $u_k\in L^r(\mathbb{R}^N)$, then $\overline{K}_{1,k}|u_k^{(\eta)}|^{r-1}\in L^s(\mathbb{R}^N)$, $|\overline{H}_{1,k}||u_k|\in L^t(\mathbb{R}^N)$. Therefore we get
\begin{align*}
 &\left(\int_{\mathbb{R}^N}|\overline{K}_{1,k}|^s|u_k^{(\eta)}|^{s(r-1)}dx\right)^{\frac{1}{s}}\left(\int_{A_{\eta,k}}|\overline{H}_{1,k}|^t|u_k|^{t}dx\right)^{\frac{1}{t}} \\
 \le&\left(\int_{\mathbb{R}^N}|\overline{K}_{1,k}|^{\frac{2N}{\alpha}}dx\right)^{\frac{\alpha}{2N}}\left(\int_{\mathbb{R}^N}|u_k^{(\eta)}|^rdx\right)^{\frac{r-1}{r}}\left(\int_{\mathbb{R}^N}|\overline{H}_{1,k}|^{\frac{2N}{\alpha}}dx\right)^{\frac{\alpha}{2N}}\left(\int_{A_{k,\eta}}|u_k|^rdx\right)^{r}.
\end{align*}
Thus, by the Lebesgue dominated convergence theorem we obtain for fixed $k\in\mathbb{R}^N$
\begin{align*}
\lim_{\eta\to\infty}\int_{A_{k,\eta}}|u_{k}|^rdx=0
\end{align*}
and then
\begin{align}\label{eq:reg-04-03}
\int_{\mathbb{R}^N}(I_{\alpha}*\overline{K}_{1,k}|u_k^{(\eta)}|^{r-2}u_k^{(\eta)})(\overline{H}_{1,k}(u_k-u_k^{(\eta)}))dx=o_{\eta}(1),
\end{align}
where $o_{\eta}(1)$ means $o_{\eta}(1)\to 0$ as $\eta\to \infty$. By the Young inequality we have
\begin{align}\label{eq:reg-04-04}
\begin{split}
\int_{\mathbb{R}^N}v|u_k^{(\eta)}|^{r-1}dx&\le\frac{1}{r}\int_{\mathbb{R}^N}|v|^rdx+\frac{r-1}{r}\int_{\mathbb{R}^N}|u_k^{(\eta)}|^{r}dx \\
&\le \frac{1}{r}\int_{\mathbb{R}^N}|v|^rdx+\frac{r-1}{r}\int_{\mathbb{R}^N}|u_{k}|^{r}dx, \\
\int_{\mathbb{R}^N}u|u_k^{(\eta)}|^{r-1}dx&\le\frac{1}{r}\int_{\mathbb{R}^N}|u|^rdx+\frac{r-1}{r}\int_{\mathbb{R}^N}|u_k^{(\eta)}|^{r}dx \\
&\le \frac{1}{r}\int_{\mathbb{R}^N}|u|^rdx+\frac{r-1}{r}\int_{\mathbb{R}^N}|u_{k}|^{r}dx.
\end{split}
\end{align}
Therefore by \eqref{eq:reg-04-00}, \eqref{eq:reg-04-01}, \eqref{eq:reg-04-02}, \eqref{eq:reg-04-03} and \eqref{eq:reg-04-04} we arrive at
\begin{align*}
&\frac{2(r-1)}{r^2}\int_{\mathbb{R}^N}|\nabla|u_k^{(\eta)}|^{\frac{r}{2}}|^2dx \\
&\ \ \le \frac{1}{r}\int_{\mathbb{R}^N}((V_1+\kappa)|u|^r+|\lambda| |v|^r)dx \\
&\ \ \ \ \ \ \ +\left(\frac{(\kappa+V_1+|\lambda|)(r-1)}{r}+C_r\right)\int_{\mathbb{R}^N}|u_k|^rdx+o_{\eta}(1).
\end{align*}
Since $|u_k^{(\eta)}|^{\frac{r}{2}}\in H^1(\mathbb{R}^N)$, by the Sobolev inequality we obtain
\begin{align}\label{eq:reg-05-01}
\begin{split}
&\frac{2(r-1)}{r^2}\mathcal {S}_N\left(\int_{\mathbb{R}^N}|u_k^{(\eta)}|^{\frac{r}{2}\frac{2N}{N-2}}dx\right)^{\frac{N-2}{N}} \\
&\ \ \ \ \ \le \frac{1}{r}\int_{\mathbb{R}^N}((V_1+\kappa)|u|^r+|\lambda| |v|^r)dx \\
&\ \ \ \ \ \ \ \ \ \ \ \ \ +\left(\frac{(\kappa+V_1+|\lambda|)(r-1)}{r}+C_r\right)\int_{\mathbb{R}^N}|u_k|^rdx+o_{\eta}(1),
\end{split}
\end{align}
where $\mathcal{S}_N$ is the Sobolev best constant defined in \eqref{eq:Sobolev-best}.  Similarly we obtain
\begin{align}\label{eq:reg-05-02}
\begin{split}
&\frac{2(r-1)}{r^2}\mathcal {S}_N\left(\int_{\mathbb{R}^N}|v_k^{(\eta)}|^{\frac{r}{2}\frac{2N}{N-2}}dx\right)^{\frac{N-2}{N}} \\
&\ \ \ \ \ \le \frac{1}{r}\int_{\mathbb{R}^N}((V_2+\kappa)|v|^r+|\lambda| |u|^r)dx \\
&\ \ \ \ \ \ \ \ \ \ \ \ +\left(\frac{(\kappa+V_2+|\lambda|)(r-1)}{r}+C_r\right)\int_{\mathbb{R}^N}|v_k|^rdx+o_{\eta}(1).
\end{split}
\end{align}
Letting $\eta\to\infty$ in \eqref{eq:reg-05-01} and \eqref{eq:reg-05-02} and using Fatou's Lemma we obtain
\begin{align}\label{eq:reg-05-03}
\begin{split}
&\frac{2(r-1)}{r^2}\mathcal{S}_N\left\{\left(\int_{\mathbb{R}^N}|u_{k}|^{\frac{r}{2}\frac{2N}{N-2}}dx\right)^{\frac{N-2}{N}}+\left(\int_{\mathbb{R}^N}|v_{k}|^{\frac{r}{2}\frac{2N}{N-2}}dx\right)^{\frac{N-2}{N}}\right\} \\
&\le\frac{1}{r}\int_{\mathbb{R}^N}((V_1+\kappa+|\lambda|)|u|^r+(V_2+\kappa+|\lambda|)|v|^r)dx \\
&\ \ \ \ +\left(\frac{(\kappa+V_1+|\lambda|)(r-1)}{r}+C_r\right)\int_{\mathbb{R}^N}|u_k|^rdx \\
&\ \ \ \ \ \ +\left(\frac{(\kappa+V_2+|\lambda|)(r-1)}{r}+C_r\right)\int_{\mathbb{R}^N}|v_k|^rdx
\end{split}
\end{align}
and 
\begin{align}\label{eq:reg-05-03-01}
\|u_{k}\|_{\frac{r}{2} \frac{2N}{N-2}}+\|v_{k}\|_{\frac{r}{2} \frac{2N}{N-2}}\le C_1(\kappa, V_1,V_2, \lambda, r, N)(\|u_k\|_r+\|v_k\|_r+\|u\|_r+\|v\|_r).
\end{align}
We also note that from \eqref{eq:reg-00-01} we have
\begin{align}\label{eq:reg-05-04}
\|u_k\|_2+\|v_k\|_2\le C_2(\kappa, V_1, V_2, \lambda)(\|u\|_2+\|v\|_2).
\end{align}

We now iterate the procedure starting from $r = 2$ to show that $u, v \in L^s(\mathbb{R}^N)$ for any $s \in \left[2, \frac{N}{\alpha} \frac{2N}{N-2} \right]$. 

Take $r=2$ in \eqref{eq:reg-05-03-01} and combine it with \eqref{eq:reg-05-04} we obtain
\begin{align*}
\|u_{k}\|_{\frac{2N}{N-2}}+\|v_{k}\|_{\frac{2N}{N-2}}\le C_3(\kappa, V_1,V_2, \lambda, N)(\|u\|_2+\|v\|_2).
\end{align*}
By Fatou's lemma we obtain
\begin{align*}
\|u\|_{\frac{2N}{N-2}}+\|v\|_{\frac{2N}{N-2}}\le C_3(\kappa, V_1,V_2, \lambda, N)(\|u\|_2+\|v\|_2)
\end{align*}
and the interpolation inequality implies that
\begin{align*}
\|u\|_s+\|v\|_s\le C_4(\kappa, V_1,V_2, \lambda, s, N)(\|u\|_2+\|v\|_2)
\end{align*}
for $s\in \left[2,\frac{2N}{N-2}\right]$.

\noindent
\textbf{Case 1: $\bm{\frac{2N}{\alpha} \le \frac{2N}{N-2}}$.} Take any $s \in \left(\frac{2N}{N-2}, \frac{N}{\alpha} \frac{2N}{N-2} \right)$. Combining \eqref{eq:reg-05-03-01} with $r = \frac{N-2}{N}s\in\left(2, \frac{2N}{\alpha}\right)$ and \eqref{eq:reg-05-04}, we obtain
\begin{align}\label{eq:reg-05-05}
\|u_k\|_s + \|v_k\|_s \le C_5(\kappa, V_1, V_2, \lambda, s, N)\left( \|u\|_2 + \|v\|_2 \right).
\end{align}
Then, by Fatou's lemma,
\begin{align*}
\|u\|_s + \|v\|_s \le C_5(\kappa, V_1, V_2, \lambda, s, N) \left( \|u\|_2 + \|v\|_2 \right).
\end{align*}
This completes the proof of Case 1. 

\ \\
\noindent
\textbf{Case 2: $\bm{\frac{2N}{N-2} < \frac{2N}{\alpha}}$.} Combining \eqref{eq:reg-05-03-01} with $r = \frac{2N}{N-2}(< \frac{2N}{\alpha})$ and \eqref{eq:reg-05-04}, we obtain
\begin{align}\label{eq:reg-05-06}
\|u_k\|_{2 \frac{N^2}{(N-2)^2}} + \|v_k\|_{2 \frac{N^2}{(N-2)^2}} \le C_6(\kappa, V_1, V_2, \lambda, N) \left( \|u\|_2 + \|v\|_2 \right).
\end{align}

If $2\frac{N^2}{(N-2)^2} \ge \frac{2N}{\alpha}$, then using \eqref{eq:reg-05-04}, \eqref{eq:reg-05-06} and the interpolation inequality, we obtain
\begin{align}\label{eq:reg-05-07}
\|u_k\|_s + \|v_k\|_s \le C_7(\kappa, V_1, V_2, \lambda, s, N)\left( \|u\|_2 + \|v\|_2 \right)
\end{align}
for all $s \in \left[2, \frac{N}{\alpha} \frac{2N}{N-2} \right] \subset \left[2, 2 \frac{N^2}{(N-2)^2} \right].$ and for some $C_2(\kappa, V_1, V_2, \lambda, s, N)>0$. 

If instead $2\frac{N^2}{(N-2)^2} < \frac{2N}{\alpha}$, then we continue the iteration with $r = 2\frac{N^2}{(N-2)^2}$. 
Since
\begin{align*}
2 \left( \frac{N}{N-2} \right)^l \to \infty \quad \text{as } l \to \infty,
\end{align*}
there exists  $l \in \mathbb{N}$ such that
\begin{align*}
2 \left( \frac{N}{N-2} \right)^l > \frac{2N}{\alpha}.
\end{align*}
Thus, we obtain
\begin{align}\label{eq:reg-05-08}
\|u_k\|_s + \|v_k\|_s \le C_8(\kappa, V_1, V_2, \lambda, s, N)\left( \|u\|_2 + \|v\|_2 \right)
\end{align}
for all
\[
s \in \left[2, \frac{N}{\alpha} \frac{2N}{N-2} \right] \subset \left[2, 2\left( \frac{N}{N-2} \right)^l \right].
\]
and some $C_3(\kappa, V_1, V_2, \lambda, s, N)>0$ . Finally, by Fatou's lemma,
\begin{align}
\|u\|_s + \|v\|_s \le C_8(\kappa, V_1, V_2, \lambda, s, N) \left( \|u\|_2 + \|v\|_2 \right)
\end{align}
for all $s \in \left[2, \frac{N}{\alpha}\frac{2N}{N-2} \right]$. This completes the proof of Case 2. 

The proof of the proposition is now complete.
\end{proof}

Let us give the regularity result for weak solutions to \eqref{eq:Choquard-system}.

\begin{Prop}\label{prop:reg-02}
Suppose that $N\ge 3$, $1+\frac{\alpha}{N}\le p,q\le \frac{N+\alpha}{N-2}$ and $(u,v)\in H^1(\mathbb{R}^N)\times H^1(\mathbb{R}^N)$ is a weak solution to \eqref{eq:Choquard-system}. Then $(u,v)\in W^{2,r}_{\mathrm{loc}}(\mathbb{R}^N)\times W^{2,r}_{\mathrm{loc}}(\mathbb{R}^N)$ for any $r\in [1,\infty)$. Moreover $(u,v)\in C^2(\mathbb{R}^N)\times C^2(\mathbb{R}^N)$. 
\end{Prop}
\begin{proof}
We first show $u,v\in W^{2,r}_{\mathrm{loc}}(\mathbb{R}^N)$. 
Setting $H_1(s)=K_1(s)=\sqrt{\mu}|s|^{p-2}s$, $H_2(s)=K_2(s)=\sqrt{\nu}|s|^{q-2}s$, \eqref{eq:Choquard-system-aux} can be written as
\begin{align*}
&-d_1\Delta u+V_1u=(I_{\alpha}*(H_1u))K_1+\lambda v\ \ \mbox{in}\ \ \mathbb{R}^N, \\
&-d_2\Delta v+V_2v=(I_{\alpha}*(H_2v))K_2+\lambda u\ \ \mbox{in}\ \ \mathbb{R}^N. 
\end{align*}
Since $1+\frac{\alpha}{N}\le p\le \frac{N+\alpha}{N-2}$ we see that
\begin{align}\label{eq:reg-07-01}
|u(x)|^p\le |u(x)|^{1+\frac{\alpha}{N}}+|u(x)|^{\frac{N+\alpha}{N-2}}, \end{align}
and
\begin{align}\label{eq:reg-07-02}
\begin{split}
&|u(x)|^{p-1}=|u(x)|^{p-1}\chi_{\{y\in\mathbb{R}^N:|u(y)|\le 1\}}(x)+|u(x)|^{p-1}\chi_{\{y\in\mathbb{R}^N:|u(x)|\ge 1\}}(x) \\
&||u|^{p-1}\chi_{\{y\in\mathbb{R}^N:|u(y)|\le 1\}}|\le |u|^{\frac{\alpha}{N}}\in L^{\frac{2N}{\alpha}}(\mathbb{R}^N),\\ 
&||u|^{p-1}\chi_{\{y\in\mathbb{R}^N:|u(y)|\ge 1\}}|\le |u|^{\frac{\alpha+2}{N-2}}\in L^{\frac{2N}{\alpha+2}}(\mathbb{R}^N).
\end{split}
\end{align}
Hence $H_1$, $K_1\in L^{\frac{2N}{\alpha}}(\mathbb{R}^N)+L^{\frac{2N}{\alpha+2}}(\mathbb{R}^N)$. Similarly we see $H_2$, $K_2\in L^{\frac{2N}{\alpha}}(\mathbb{R}^N)+L^{\frac{2N}{\alpha+2}}(\mathbb{R}^N)$. By Proposition \ref{prop:reg}, $u$, $v\in L^r(\mathbb{R}^N)$ for $r\in [2, \frac{N}{\alpha} \frac{2N}{N-2})$. By \eqref{eq:reg-07-01} and \eqref{eq:reg-07-02}, $|u|^{p}, |v|^q\in L^s(\mathbb{R}^N)$ for any $s\in [\frac{2N}{N+\alpha}, \frac{N}{\alpha} \frac{2N}{N+\alpha})$. Since $\frac{2N}{N+\alpha}<\frac{N}{\alpha}<\frac{N}{\alpha} \frac{2N}{N+\alpha}$ we have $I_{\alpha}*|u|^p, I_{\alpha}*|v|^q\in L^{\infty}(\mathbb{R}^N)$. 
Therefore we obtain
\begin{align*}
|-d_1\Delta u+V_1u|\le C(|u|^{\frac{\alpha}{N}}+|u|^{\frac{2+\alpha}{N-2}}+|v|),\ \ |-d_2\Delta v+V_2v|\le C(|v|^{\frac{\alpha}{N}}+|v|^{\frac{2+\alpha}{N-2}}+|u|).
\end{align*}
By the classical bootstrap argument for subcritical local problem on bounded domains we deduce that $u$, $v\in W^{2,r}_{\mathrm{loc}}(\mathbb{R}^N)$ for any $r\ge 1$. 

We next prove that $u,v \in C^2(\mathbb{R}^N)$.  
Although the result can be proved in the same manner as in the proof of Claim~3 in \cite[Proposition 4.1]{Moroz-Schaftingen}, we include the proof here for completeness.

By the Sobolev embedding theorem (see, for example, \cite[Corollary 9.13]{Brezis}), we have $u \in C^{1,1 - \frac{N}{r}}_{\mathrm{loc}}(\mathbb{R}^N)$ for any $r > N$. We shall show that if $0 < \gamma < \min\{1, p - 1\}$, then $u, v\in C^{2,\gamma}_{\mathrm{loc}}(\mathbb{R}^N)$.

Take $r > 1$ sufficiently large such that
\begin{align*}
\gamma < \min\{1, p - 1\}\left(1 - \frac{N}{r}\right) =: \theta.
\end{align*}
Let $R > 0$ and define $B_R = B_R(0)$, $B_{R+1} = B_{R+1}(0)$, and $B_{R+3} = B_{R+3}(0)$. Choose $\eta \in C_0^\infty(\mathbb{R}^N)$ such that $\eta \equiv 1$ on $B_1(0)$, $\eta \equiv 0$ on $\mathbb{R}^N \setminus B_2(0)$, and $0 \le \eta \le 1$ on $\mathbb{R}^N$. Write
\begin{align*}
I_\alpha = (1 - \eta) I_\alpha + \eta I_\alpha.
\end{align*}
Since $u \in L^s(\mathbb{R}^N)$ for any $s \ge 1$, and $\nabla \big[(1 - \eta) I_\alpha\big]$ is bounded on $\mathbb{R}^N$, it follows that
\begin{align*}
((1 - \eta)I_{\alpha}) * |u|^p \in C^1(\mathbb{R}^N),
\end{align*}
and indeed, as shown in~\cite{Moroz-Schaftingen}, we have $((1 - \eta)I_{\alpha}) * |u|^p \in C^\infty(\mathbb{R}^N)$.

Since the function $t \mapsto |t|^p$ is locally Lipschitz continuous, it follows that $|u|^p \in C^{0,\gamma}(B_{R+3})$. Moreover, since the function $t \mapsto |t|^{p-2}t$ belongs to $C^{0,\min\{1, p - 1\}}_{\mathrm{loc}}(\mathbb{R})$, we deduce that $|u|^{p-2}u \in C^{0,\theta}(B_{R+3})$, and hence $|u|^{p-2}u \in C^{0,\gamma}(B_{R+3})$.

Next, since $\eta I_\alpha \in L^1(\mathbb{R}^N)$ and $|u|^p \in C^{0,\gamma}(B_{R+3})$, we obtain
\begin{align*}
|(\eta I_\alpha * |u|^p)(x) - (\eta I_\alpha * |u|^p)(z)|
&= A_\alpha \left| \int_{B_2(0)} \frac{\eta(y)}{|y|^{N - \alpha}} \left( |u(x - y)|^p - |u(z - y)|^p \right) \, dy \right| \\
&\le A_\alpha \|u\|_{C^{0,\gamma}(B_{R+3})} |x - z|^\gamma \int_{B_2(0)} \frac{1}{|y|^{N - \alpha}} \, dy\ \ \ \text{for}\ \ \ x,z\in B_{R+1}.
\end{align*}
Therefore, $(\eta I_\alpha) * |u|^p \in C^{0,\gamma}(B_{R+1})$, and so
\begin{align*}
(I_\alpha * |u|^p) |u|^{p - 2}u+\lambda v \in C^{0,\gamma}(B_{R+1}).
\end{align*}

Finally, by the Schauder estimate (see~\cite[Theorem 4.6]{Gilberg-Trudinger}), we conclude that $u \in C^{2,\gamma}(B_R)$. Similarly we can obtain $v\in C^{2,\gamma}_{\mathrm{loc}}(\mathbb{R}^N)$.   
The proof is complete.

\end{proof}

\subsection{Proof of Lemma \ref{Lem:Pohozaev}}
Now we are in a position to prove the Pohozaev identity. 
\begin{proof}[Proof of Lemma \ref{Lem:Pohozaev}]
Although we follow the arguments in~\cite{Moroz-Schaftingen, Moroz-Schaftingen-3}, we provide the proof here for the reader’s convenience.

Let $(u,v)\in H$ be a weak solution to \eqref{eq:NKC}, and set $d_1=a_1+b_1\int_{\mathbb{R}^3}|\nabla u|^2dy$ and $d_2=a_2+b_2\int_{\mathbb{R}^3}|\nabla v|^2dy$. Then $(u,v)\in H$ is a weak solution to
\begin{align*}
\begin{cases}
-d_1\Delta u+V_1u=\mu(I_{\alpha}*|u|^p)|u|^p+\lambda v\ \ \mbox{for}\ \ x\in\mathbb{R}^3, \\
-d_2\Delta v+V_2v=\nu(I_{\alpha}*|v|^q)|v|^q+\lambda u\ \ \mbox{for}\ \ x\in\mathbb{R}^3.
\end{cases}
\end{align*}
By Proposition \ref{prop:reg-02}, we have $u, v \in H^2_{\mathrm{loc}}(\mathbb{R}^3)$.  
Let $\varphi \in C^{\infty}_0(\mathbb{R}^3)$ be such that $\varphi \equiv 1$ on $B_1(0)$.  
For $\rho > 0$, set
\begin{align*}
u_{\rho}(x) := \varphi(\rho x)\, x \cdot \nabla u(x), \quad
v_{\rho}(x) := \varphi(\rho x)\, x \cdot \nabla v(x).
\end{align*}
Since $u, v \in H^2_{\mathrm{loc}}(\mathbb{R}^3)$ and $\varphi \in C^{\infty}_0(\mathbb{R}^3)$, it follows that $(u_\rho, v_\rho) \in H$.  
From the fact that $\langle I'(u,v), (u_\rho, v_\rho) \rangle = 0$, we obtain
\begin{align}\label{eq:poho-proof}
\begin{split}
0 =\ 
    & a_1 \int_{\mathbb{R}^3} \nabla u \cdot \nabla u_\rho \, dx
    + a_2 \int_{\mathbb{R}^3} \nabla v \cdot \nabla v_\rho \, dx
    + \int_{\mathbb{R}^3} V_1 u u_\rho \, dx
    + \int_{\mathbb{R}^3} V_2 v v_\rho \, dx \\
    & + b_1 \left( \int_{\mathbb{R}^3} |\nabla u|^2 \, dx \right) \int_{\mathbb{R}^3} \nabla u \cdot \nabla u_\rho \, dx
    + b_2 \left( \int_{\mathbb{R}^3} |\nabla v|^2 \, dx \right) \int_{\mathbb{R}^3} \nabla v \cdot \nabla v_\rho \, dx \\
    & - \mu \int_{\mathbb{R}^3} (I_\alpha * |u|^p)\, |u|^{p-2} u\, u_\rho \, dx
    - \nu \int_{\mathbb{R}^3} (I_\alpha * |v|^q)\, |v|^{q-2} v\, v_\rho \, dx \\
    & - \lambda \int_{\mathbb{R}^3} u v_\rho \, dx
    - \lambda \int_{\mathbb{R}^3} v u_\rho \, dx.
\end{split}
\end{align}

Since $u \in H^2_{\mathrm{loc}}(\mathbb{R}^3)$ and $\varphi \in C_0^\infty(\mathbb{R}^3)$, we may apply integration by parts on bounded domains with smooth boundary. Then we obtain
\begin{align*}
\int_{\mathbb{R}^3} \nabla u \cdot \nabla u_\rho \, dx
&= \rho \int_{\mathbb{R}^3} \nabla u \cdot \nabla \varphi(\rho x)\, (x \cdot \nabla u)\, dx
  + \int_{\mathbb{R}^3} \varphi(\rho x)\left[|\nabla u|^2 + x \cdot \nabla \left( \frac{|\nabla u|^2}{2} \right) \right] dx \\
&= \rho \int_{\mathbb{R}^3} \nabla u \cdot \nabla \varphi(\rho x)\, (x \cdot \nabla u)\, dx
  + \int_{\mathbb{R}^3} \varphi(\rho x)\, |\nabla u|^2\, dx \\
&\quad - \int_{\mathbb{R}^3} \rho x \cdot \nabla \varphi(\rho x)\, \frac{|\nabla u|^2}{2}\, dx
  - \frac{3}{2} \int_{\mathbb{R}^3} \varphi(\rho x)\, |\nabla u|^2\, dx \\
&= -\frac{1}{2} \int_{\mathbb{R}^3} \varphi(\rho x)\, |\nabla u|^2\, dx
  + \int_{\mathbb{R}^3} \left[
      \rho (\nabla u \cdot \nabla \varphi(\rho x))(x \cdot \nabla u)
      - \rho x \cdot \nabla \varphi(\rho x)\, \frac{|\nabla u|^2}{2}
    \right] dx.
\end{align*}

Since $\varphi \equiv 1$ on $B_1(0)$, $\varphi \in C_0^\infty(\mathbb{R}^3)$, and $|\nabla u| \in L^2(\mathbb{R}^3)$, the dominated convergence theorem implies
\begin{align*}
\lim_{\rho \to +0} \int_{\mathbb{R}^3} \nabla u \cdot \nabla u_\rho \, dx = -\frac{1}{2} \int_{\mathbb{R}^3} |\nabla u|^2 \, dx.
\end{align*}
Similarly, we have
\begin{align*}
\lim_{\rho \to +0} \int_{\mathbb{R}^3} \nabla v \cdot \nabla v_\rho \, dx = -\frac{1}{2} \int_{\mathbb{R}^3} |\nabla v|^2 \, dx.
\end{align*}

Next, by integration by parts we get
\begin{align*}
\int_{\mathbb{R}^3} u u_\rho \, dx
&= \int_{\mathbb{R}^3} u\, \varphi(\rho x)\, (x \cdot \nabla u)\, dx \\
&= -\frac{3}{2} \int_{\mathbb{R}^3} \varphi(\rho x)\, |u|^2\, dx
   - \int_{\mathbb{R}^3} \rho (x \cdot \nabla \varphi(\rho x))\, |u|^2\, dx.
\end{align*}
Hence, again by the dominated convergence theorem
\begin{align*}
\lim_{\rho \to +0} \int_{\mathbb{R}^3} u u_\rho \, dx = -\frac{3}{2} \int_{\mathbb{R}^3} |u|^2 \, dx,
\quad
\lim_{\rho \to +0} \int_{\mathbb{R}^3} v v_\rho \, dx = -\frac{3}{2} \int_{\mathbb{R}^3} |v|^2 \, dx.
\end{align*}

Similarly, by integration by parts
\begin{align*}
\int_{\mathbb{R}^3} u v_\rho \, dx + \int_{\mathbb{R}^3} v u_\rho \, dx
= -3 \int_{\mathbb{R}^3} \varphi(\rho x)\, u v \, dx
  - 2\int_{\mathbb{R}^3} \rho u v (x \cdot \nabla \varphi(\rho x)) \, dx.
\end{align*}
Letting $\rho \to +0$, we conclude
\begin{align*}
\lim_{\rho \to +0} \left( \int_{\mathbb{R}^3} u v_\rho \, dx + \int_{\mathbb{R}^3} v u_\rho \, dx \right)
= -3 \int_{\mathbb{R}^3} u v \, dx.
\end{align*}

Finally we consider the nonlocal term. Since $I_{\alpha}(x-y)=I_{\alpha}(y-x)$, we obtain
\begin{align*}
 &\ \int_{\mathbb{R}^3}(I_{\alpha}*|u|^p)|u|^{p-2}uu_{\rho}dx \\
=&\ \int_{\mathbb{R}^3}\int_{\mathbb{R}^3}|u(y)|^pI_{\alpha}(x-y)\varphi(\rho x)x\cdot\nabla\left(\frac{|u|^p}{p}\right)(x)dxdy \\
=&\ \frac{1}{2}\int_{\mathbb{R}^3}\int_{\mathbb{R}^3}I_{\alpha}(x-y)\left\{|u(y)|^p\varphi(\rho x)x\cdot\nabla\left(\frac{|u|^p}{p}\right)(x)+|u(x)|^p\varphi(\rho y)y\cdot\nabla\left(\frac{|y|^p}{p}\right)(y)\right\}dxdy
\end{align*}
By the Fubini theorem and integration by parts we obtain
\begin{align}\label{eq:poho-nonlocal}
 &\ \int_{\mathbb{R}^3}(I_{\alpha}*|u|^p)|u|^{p-2}uu_{\rho}dx \\
 =&\ -\int_{\mathbb{R}^3}\int_{\mathbb{R}^3}|u(y)|^pI_{\alpha}(x-y)(3\varphi(\rho x)+x\cdot\nabla\varphi(\rho x))\frac{|u(x)|^p}{p}dx \\ 
  &\ \ +\frac{3-\alpha}{2p}\int_{\mathbb{R}^3}\int_{\mathbb{R}^3}|u(y)|^pI_{\alpha}(x-y)\frac{(x-y)\cdot (x\varphi(\rho x)-y\varphi(\rho y))}{|x-y|^2}|u(x)|^pdxdy
\end{align}
Now we estimate 
\begin{align*}
\frac{(x-y)\cdot (x\varphi(\rho x)-y\varphi(\rho y))}{|x-y|^2}.
\end{align*}
We may assume that $\mathrm{supp}\varphi\subset B_R(0)$ for $R>0$. If $|y|\le R/\rho$, then
\begin{align*}
 \left|\frac{(x-y)\cdot (x\varphi(\rho x)-y\varphi(\rho y))}{|x-y|^2}\right| 
&=\left|\frac{(x-y)\cdot\{(x-y)\varphi(\rho x)+y(\varphi(\rho x)-\varphi(\rho y))\}}{|x-y|^2}\right| \\
&\le |\varphi(\rho x)|+\rho |y|\|\nabla\varphi\|_{\infty}\le\|\varphi\|_{\infty}+R\|\nabla\varphi\|_{\infty}.
\end{align*}
Similarly when $|x|\le R/\rho$, the same estimate will be obtained. If $|x|>R/\rho$ and $|y|>R/\rho$, then $x\varphi(\rho x)-y\varphi(\rho y)=0$. Therefore
\begin{align*}
\left|\frac{(x-y)\cdot (x\varphi(\rho x)-y\varphi(\rho y))}{|x-y|^2}\right|\le \|\varphi\|_{\infty}+R\|\nabla\varphi\|_{\infty}\ \ \mbox{for}\ \ (x,y)\in\mathbb{R}^3\times\mathbb{R}^3.
\end{align*}
Therefore we can use the Lebesgue dominated convergence theorem in \eqref{eq:poho-nonlocal} to obtain
\begin{align*}
\lim_{\rho\to +0}\int_{\mathbb{R}^3}(I_{\alpha}*|u|^p)|u|^{p-2}uu_{\rho}dx=-\frac{3+\alpha}{2p}\int_{\mathbb{R}^3}(I_{\alpha}*|u|^p)|u|^{p}dx.
\end{align*}
Letting $\rho\to +0$ in \eqref{eq:poho-proof} we obtain \eqref{eq:Pohozaev-id}. 
\end{proof}

\section{The Nehari--Pohozaev manifold}
    As we state in Section 1 we obtain the ground state solution as a minimizer of the minimization problem over the Nehari--Pohozaev manifold. In this section we prepare the Nehari--Pohozaev manifold and its properties. 
    \begin{Lem}     \label{Lem:unbounded}
      Assume that \ref{assumption:potential} holds and let $(3+\alpha)/3\le p\le q\le 3+\alpha$, then $I$ is not bounded from below.
    \end{Lem}
    
\begin{proof}
Let $(u,v)\in H\setminus\{(0,0)\}$. By a direct calculation we have
\begin{align*}
I(u^t, v^t)&=\frac{1}{2}t^4(a_1\|\nabla u\|_2^2+a_2\|\nabla v\|_2^2)+\frac{1}{2}t^8(V_1\|u\|_2^2+V_2\|v\|_2^2)+\frac{1}{4}t^8(b_1\|\nabla u\|_2^4+b_2\|\nabla v\|_2^4) \\
   &\quad -\frac{\mu}{2p}
   t^{2(p+3+\alpha)}\int_{\mathbb{R}^3}(I_{\alpha}*|u|^p)|u|^pdx-\frac{\nu}{2q}t^{2(q+3+\alpha)}\int_{\mathbb{R}^3}(I_{\alpha}*|v|^q)|v|^qdx-\lambda t^8\int_{\mathbb{R}^3}uvdx.
\end{align*}
Since $2(p+\alpha+3), 2(q+\alpha+3)>8$, it is easily seen that $I(u^t, v^t)\to -\infty$ as $t\to\infty$. 
\end{proof}

To prove the Nehari--Pohozaev manifold, which will be defined below, is nonempty we need the following lemma.

\begin{Lem}\label{Lem:mpstructure}
Let $C_1$, $C_2$, $C_3$, $C_4$ are positive constants, and $p$, $q$ are numbers satisfying $3+\alpha\le p\le q\le 3+\alpha$.  If $f(t)=C_1t^4+C_2t^8-C_3t^{2(p+3+\alpha)}-C_4t^{2(q+3+\alpha)}$, then $f$ has a unique critical point which corresponds to its maximum over $(0,\infty)$. 
\end{Lem}
\begin{proof}
We can prove this lemma in the same manner as the proof of \cite[Lemma 2.3]{Tatsuya} by investigating $f'$, $f''$, $\cdots$, $f^{(8)}$. 
\end{proof}
   Let $(u, v) \in H \setminus \{(0, 0)\}$ be a critical point of $I$. For $t > 0$, set
    \begin{align}
        \begin{split}
            \zeta(t) &\coloneqq I(u^t, v^t)      \label{eq:zeta} \\
            &= \frac{1}{2}t^4(a_1\|\nabla u\|_2^2 + a_2\|\nabla v\|_2^2) + \frac{1}{2}t^8\left(V_1\|u\|_2^2 + V_2\|v\|_2^2\right) + \frac{1}{4}t^8(b_1\|\nabla u\|_2^4 + b_2\|\nabla v\|_2^2) \\
            & - \frac{\mu}{2p}t^{2(p +\alpha+ 3)}\int_{\mathbb{R}^3}(I_{\alpha}*|u|^p)|u|^pdx - \frac{\nu}{2q}t^{2(q +\alpha+ 3)}\int_{\mathbb{R}^3}(I_{\alpha}*|v|^q)|v|^qdx - \lambda t^8 \int_{\mathbb{R}^3} uv\,dx.
        \end{split}
    \end{align}
    By Lemma \ref{Lem:mpstructure}, $\zeta$ has a unique critical point $t_1 > 0$ corresponding to its maximum. Since $(u, v)$ is a critical point of $I$, we see that $t_1 = 1$ and
    \begin{align*}
        \zeta^{\prime}(1) &=2(a_1\|\nabla u\|_2^2 + a_2\|\nabla v\|_2^2) + 4(V_1\|u\|_2^2 + V_2\|v\|_2^2)+2(b_1\|\nabla u\|_2^4 + b_2\|\nabla v\|_2^4)\\
        &- \frac{p +\alpha+ 3}{p}\mu\int_{\mathbb{R}^3}(I_{\alpha}*|u|^p)|u|^pdx - \frac{q +\alpha+ 3}{q}\nu\int_{\mathbb{R}^3}(I_{\alpha}*|v|^q)|v|^q - 8\lambda \int_{\mathbb{R}^3} uv\,dx = 0.
    \end{align*}
    From the above observation we define the functional $J \colon H \rightarrow \mathbb{R}$ by
    \begin{align}\label{eq:NP functional}
        \begin{split}
            J(u, v) &\coloneqq 2(a_1\|\nabla u\|_2^2 + a_2\|\nabla v\|_2^2) + 4(V_1\|u\|_2^2 + V_2\|v\|_2^2) \\
                &\quad + 2(b_1\|\nabla u\|_2^4 + b_2\|\nabla v\|_2^4)- \frac{p +\alpha+ 3}{p}\mu\int_{\mathbb{R}^3}(I_{\alpha}*|u|^p)|u|^pdx \\
            &\quad - \frac{q +\alpha+3}{q}\nu\int_{\mathbb{R}^3}(I_{\alpha}*|v|^q)|v|^qdx - 8\lambda \int_{\mathbb{R}^3} uv\,dx
        \end{split}
    \end{align}
    and the Nehari--Pohozaev manifold $\mathcal{M}$ by
    \begin{align}
        \mathcal{M} \coloneqq \{(u, v) \in H \setminus \{(0, 0)\} \mid J(u, v) = 0\}.   \label{eq:NPmanifold}
    \end{align}
    It is clear that
    \begin{align}\label{eq:NP-functional-0}
    J(u, v) = \langle I^{\prime}(u, v), (u, v) \rangle + 2P(u, v).
    \end{align}
From \eqref{eq:NPmanifold} and \eqref{eq:NP-functional-0}, it follows that any nontrivial critical point of $I$ lies in $\mathcal{M}$.

  The following lemma gives very important property of the Nehari--Pohozaev manifold.
  \begin{Lem}     \label{Lem:positive on NP}
        Assume that \ref{assumption:potential} holds. There exists $\underline{C}> 0$ such that 
        \begin{align}
            \|(u, v)\| \ge \underline{C} \label{eq:positive on NP}
        \end{align}
        for all $(u,v)\in\mathcal{M}$.
    \end{Lem}
    \begin{proof}We first note that 
            \begin{align*}
            2\sqrt{V_1V_2}|u||v| \le V_1u^2 + V_2v^2.
        \end{align*}
        holds and \ref{assumption:potential}, we see that
        \begin{align}\label{eq:-lambda}
        \begin{split}
            \int_{\mathbb{R}^3}(V_1u^2+V_2v^2)dx-2\lambda \int_{\mathbb{R}^3} uv\,dx &\ge 
            \int_{\mathbb{R}^3}(V_1u^2+V_2v^2)dx-2\delta \int_{\mathbb{R}^3} \sqrt{V_1V_2}|u||v|\,dx \\
            &\ge (1-\delta) \int_{\mathbb{R}^3} (V_1u^2 + V_2v^2)\,dx
        \end{split}
        \end{align}
        Since $(u, v) \in \mathcal{M}$, we have $J(u, v) = 0$. Thus, using \eqref{eq:-lambda} 
        we obtain
        \begin{align}\label{eq:estimate-H1-on-M}
        \begin{split}
            2\left(1 - \delta\right)\|(u, v)\|^2 
            &=2(1-\delta)(a_1\|\nabla u\|_2^2 + a_2\|\nabla v\|_2^2) + 2(1-\delta)\left(V_1\|u\|_2^2 + V_2\|v\|_2^2\right) \\
            &\le 2(a_1\|\nabla u\|_2^2 + a_2\|\nabla v\|_2^2)+4(1-\delta)\left(V_1\|u\|_2^2 + V_2\|v\|_2^2\right) \\ 
            &\le 2(a_1\|\nabla u\|_2^2 + a_2\|\nabla v\|_2^2)+4\left(V_1\|u\|_2^2 + V_2\|v\|_2^2-2\lambda\int_{\mathbb{R}^3}uvdx\right) \\
            &\quad +2(b_1\|\nabla u\|_2^4 + b_2\|\nabla v\|_2^4) \\
            &= \frac{p + \alpha+3}{p}\mu\int_{\mathbb{R}^3}(I_{\alpha}*|u|^p)|u|^pdx + \frac{q +\alpha+ 3}{q}\nu\int_{\mathbb{R}^3}(I_{\alpha}*|v|^q)|v|^qdx
            \end{split}
        \end{align}
        By \eqref{eq:convolution-term-est} we obtain
        \begin{align*}
        2(1-\delta)\|(u,v)\|^2\le  C_1(\|(u,v)\|^{2p}+\|(u,v)\|^{2q}),
        \end{align*}
        where $C_1=C_1(p,q,\alpha, a_1, a_2, V_1, V_2,\mu,\nu)>0$ is a constant.
        Hence, we have
        \begin{align*}
            0 < \frac{2(1 -\delta)}{C_1} &\le \|(u, v)\|^{2p - 2} + \|(u, v)\|^{2q - 2} \\
            &
            \begin{cases}
                \le 2\|(u, v)\|^{2p - 2}, & \text{if } \|(u, v)\| \le 1, \\
                \le 2\|(u, v)\|^{2q - 2}, & \text{if } \|(u, v)\| \ge 1
            \end{cases}
        \end{align*}
        and so \eqref{eq:positive on NP} holds if we set
        \begin{align*}
            \underline{C}\coloneqq \min \left\{\left(\frac{1 -\delta}{C_1}\right)^{\frac{1}{2p - 2}}, \left(\frac{(1-\delta)}{C_1}\right)^{\frac{1}{2q - 2}}\right\}.
        \end{align*}
        The proof of Lemma \ref{Lem:positive on NP} has been completed. 
    \end{proof}

The following lemma also plays a crucial role in proving that the infimum of $I(u)$ on $\mathcal{M}$ is attained and that it is a critical value of $I$.
\begin{Lem}\label{lem:energy-ineq}
Assume that \ref{assumption:potential} holds. Then for any $(u,v)\in H$ and $t>0$, the following inequality holds:
\begin{align}\label{eq:energy-ineq-01}
\begin{split}
I(u,v)=&I(u^t,v^t)+\frac{1-t^8}{8}J(u,v)+\frac{(1-t^4)^2}{4}(a_1\|\nabla u\|_2^2+a_2\|\nabla v\|_2^2) \\
 &\ \ \ +\left\{\frac{p+\alpha+3}{8p}(1-t^8)-\frac{1-t^{2(p+\alpha+3)}}{2p}\right\}\mu\int_{\mathbb{R}^3}(I_{\alpha}*|u|^p)|u|^pdx \\
 &\ \ \ +\left\{\frac{q+\alpha+3}{8q}(1-t^8)-\frac{1-t^{2(p+\alpha+3)}}{2q}\right\}\nu\int_{\mathbb{R}^3}(I_{\alpha}*|v|^q)|v|^qdx
\end{split}
\end{align}
In particular we have
\begin{align}\label{eq:energy-ineq-02}
I(u,v)\ge I(u^t,v^t)+\frac{1-t^8}{8}J(u,v)+\frac{(1-t^4)^2}{4}(a_1\|\nabla u\|_2^2+a_2\|\nabla v\|_2^2)
\end{align}
\end{Lem}
\begin{proof}
\eqref{eq:energy-ineq-01} is obtained by a straightforward calculation. To obtain \eqref{eq:energy-ineq-02} we only need to verify that
\begin{align*}
g(t):=\frac{r+\alpha+3}{8r}(1-t^8)-\frac{1-t^{2(r+\alpha+3)}}{2r}>g(1)=0\ \ \mbox{for}\ \ t\in [0, 1)\cup (1,\infty).
\end{align*}
holds if $0<\alpha<3$ and $(3+\alpha)/3\le r\le 3+\alpha$. 
\end{proof}
The following corollary follows immediately from the lemma above.
\begin{Cor}\label{Cor:maximum}
   Assume that \ref{assumption:potential} holds. Then for any $(u,v)\in\mathcal{M}$ it holds that
   \begin{align*}
    I(u,v)=\max_{t>0}I(u^t, v^t).
   \end{align*}
\end{Cor}
\begin{proof}
It is clear that $I(u,v)\le \max_{t>0} I(u^t, v^t)$. Suppose that $(u,v)\in \mathcal{M}$. Since $J(u,v)=0$ by Lemma \ref{lem:energy-ineq} we have $I(u,v)\ge I(u^t,v^t)$ for any $t>0$. Therefore $I(u,v)\ge \max_{t>0}I(u^t,v^t)$ holds. The proof is now complete.
\end{proof}

We now show that $\mathcal{M} \ne \emptyset$ and provide a characterization of $\mathcal{M}$ in terms of the function $I(u^t, v^t)$ of $t$. 

\begin{Lem}\label{lem:tuv_in_M}
Assume that \ref{assumption:potential} holds. Then, for any $(u,v) \in H \setminus \{(0,0)\}$, there exists a unique $t = t(u,v) > 0$ such that $(u^{t(u,v)}, v^{t(u,v)}) \in \mathcal{M}$.
\end{Lem}

\begin{proof}
   Let $(u, v) \in H \setminus \{(0, 0)\}$ be fixed and define $\zeta(t) = I(u^t, v^t)$. Then we see that
        \begin{align*}
            \zeta^{\prime}(t) &= 2t^3(a_1\|\nabla u\|_2^2 + a_2\|\nabla v\|_2^2) + 4t^7\left(V_1\|u\|_2^2 + V_2\|v\|_2^2\right) + 2t^7(b_1\|\nabla u\|_2^4 + b_2\|\nabla v\|_2^2) \\
            &\quad - \frac{p +\alpha+ 3}{p}\mu t^{2(p +\alpha+ 3)-1}\int_{\mathbb{R}^3}(I_{\alpha}*|u|^p)|u|^pdx \\
            &\quad - \frac{q +\alpha+ 3}{q}t^{2(q + \alpha+3)-1}\nu\int_{\mathbb{R}^3}(I_{\alpha}*|v|^q)|v|^qdx- 8\lambda t^7 \int_{\mathbb{R}^3} uv\,dx \\
            &= t^{-1}J(u^t, v^t).
        \end{align*}
        By Lemma \ref{Lem:mpstructure}, $\zeta$ has a unique critical point $t=t(u,v) > 0$ corresponding to its maximum. Hence, we obtain $J(u^{t(u,v)}, v^{t(u,v)}) = 0$, which means $(u^{t(u,v)}, v^{t(u,v)}) \in \mathcal{M}$.
    \end{proof}

Combining Corollary \ref{Cor:maximum} and Lemma \ref{lem:tuv_in_M}, we get the following identity.
\begin{align}\label{eq:Energy level}
m:=\inf_{(u,v)\in\mathcal{M}}I(u,v)=\inf_{(u,v)\in H\setminus \{(0,0)\}}\max_{t>0}I(u^t, v^t).
\end{align}
    
We now prove $m>0$.  
    \begin{Lem}     \label{Lem:inf is positive}
         Assume that \ref{assumption:potential} holds. Then $m > 0$.
    \end{Lem}
   \begin{proof}Without loss of generality we may assume that $p\le q$.  Let $(u, v) \in \mathcal{M}$. Then we have $J(u,v)=0$. From the definition of $I(u,v)$ in \eqref{eq:energy} and $J(u,v)$ in \eqref{eq:NP functional} 
       we obtain 
    \begin{align}\label{eq:I is positive}
        \begin{split}
 &\ I(u, v) \\
=&\ I(u,v)-\frac{1}{2(p+3+\alpha)}J(u,v) \\
=&\ \frac{p+\alpha+1}{2(p+3+\alpha)}(a_1\|\nabla u\|_2^2 + a_2\|\nabla v\|_2^2)+\frac{p+\alpha-1}{2(p+\alpha+3)}\left\{\int_{\mathbb{R}^3}(V_1u^2+V_2v^2-2\lambda uv)dx\right\} \\
            &\quad+\frac{p+\alpha-1}{4(p+\alpha+3)}(b_1\|\nabla u\|_2^4+b_2\|\nabla v\|_2^4)+\nu\frac{q-p}{2q(p+3+\alpha)}\int_{\mathbb{R}^3}(I_{\alpha}*|v|^q)|v|^qdx \\
            \ge&\ \frac{p+\alpha+1}{2(p+\alpha+3)}(a_1\|\nabla u\|_2^2+a_2\|\nabla v\|_2^2)+\frac{p+\alpha-1}{2(p+\alpha+3)}\left\{\int_{\mathbb{R}^3}(V_1u^2+V_2v^2-2\lambda uv)dx\right\}  \\
                   \ge&\ \frac{p+\alpha-1}{2(p+\alpha+3)}(1-\delta)\|(u,v)\|^2.
\end{split}
       \end{align}
       Therefore we obtain
       \begin{align*}     
              m=\inf_{(u,v)\in\mathcal{M}}I(u,v)\ge(1-\delta)\frac{p+\alpha-1}{2(p+\alpha+3)}\underline{C}^2>0
       \end{align*}
       and Lemma \ref{Lem:inf is positive} has been proved.
   \end{proof}
 The next lemma also plays a very important role in Section 4 to prove that $m$ is achieved.
   
   \begin{Lem}     \label{Lem:weak limit identity}
        Assume that \ref{assumption:potential} holds. If $u_n \rightharpoonup u$ and $v_n \rightharpoonup v$ in $H^1(\mathbb{R}^3)$, then passing to a subsequence, we have the following identities:
        \begin{align}
            \begin{split}
                I(u_n, v_n) &= I(u, v) + I(u_n - u, v_n - v) \\
                &\quad + \frac{1}{2}\left(b_1\|\nabla u\|_2\|\nabla (u_n - u)\|_2^2 + b_2\|\nabla v\|_2^2\|\nabla (v_n - v)\|_2^2\right) + o(1),  \label{eq:Identity of I} 
            \end{split} \\
            \begin{split}
                \langle I^{\prime}(u_n, v_n), (u_n, v_n) \rangle &= \langle I^{\prime}(u, v), (u, v) \rangle + \langle I^{\prime}(u_n - u, v_n - v), (u_n - u, v_n - v) \rangle \\
                &\quad + 2\left(b_1\|\nabla u\|_2^2\|\nabla (u_n - u)\|_2^2 + b_2\|\nabla v\|_2^2\|\nabla (v_n - v)\|_2^2\right) + o(1),   \label{eq:Identity of I'}
            \end{split}
        \end{align}
        and
        \begin{align} \label{eq:Identity of J}
            \begin{split}
                J(u_n, v_n) &= J(u, v) + J(u_n - u, v_n - v) \\
                &\quad + 4\left(b_1\|\nabla u\|_2^2\|\nabla (u_n - u)\|_2^2 + b_2\|\nabla v\|_2^2\|\nabla (v_n - v)\|_2^2\right) + o(1).  
            \end{split}
        \end{align}
        Here $o(1)$ means that $o(1) \to 0$ as $n \to \infty$.
    \end{Lem}
  \begin{proof}
        Set
        \begin{align}
            & I_1(u, v) \coloneqq a_1\|\nabla u\|_2^2 + a_2\|\nabla v\|_2^2 + V_1\|u\|_2^2 + V_2|v|_2^2, \label{eq:I_1} \\
            & I_2(u, v) \coloneqq b_1\|\nabla u\|_2^4 + b_2\|\nabla v\|_2^4, \label{eq:I_2} \\
            & I_3(u, v) \coloneqq \frac{\mu}{2p}\int_{\mathbb{R}^3}(I_{\alpha}*|u|^p)|u|^pdx + \frac{\nu}{2q}\int_{\mathbb{R}^3}(I_{\alpha}*|v|^q)|v|^qdx, \label{eq:I_3} \\
            & I_4(u, v) \coloneqq \lambda \int_{\mathbb{R}^3} uv\,dx.   \label{eq:I_4}
        \end{align}
        Then $I(u,v)$ can be written by
        \begin{align*}
   I(u,v)=\frac{1}{2}I_1(u,v)+\frac{1}{4}I_2(u,v)-I_3(u,v)-I_4(u,v).
\end{align*}

Let $\omega_n \coloneqq u_n - u$ and $\sigma_n \coloneqq v_n - v$. Then along a subsequence,
        \begin{empheq}[left = \empheqlbrace]{align*}
            & (\omega_n, \sigma_n) \rightharpoonup (0, 0) \text{ weakly in } H, \\
               & (\omega_n, \sigma_n) \rightharpoonup (0, 0) \text{ weakly in } D^{1,2}(\mathbb{R}^3)\times D^{1,2}(\mathbb{R}^3), \\
            & (\omega_n, \sigma_n) \to (0, 0) \text{ strongly in } L_{\mathrm{loc}}^s(\mathbb{R}^3) \times L_{\mathrm{loc}}^s(\mathbb{R}^3) \text{ for all } s \in [1, 6), \\
            & (\omega_n, \sigma_n) \to (0, 0) \text{ almost everywhere in } \mathbb{R}^3.
        \end{empheq}
        Noting that
        \begin{align*}
        &\|\nabla u_n\|^2= \|\nabla \omega_n + \nabla u\|^2= \|\nabla \omega_n\|^2 + 2\nabla \omega_n \cdot \nabla u + \|\nabla u\|^2, \\
        &\|\nabla v_n\|^2= \|\nabla \sigma_n+\nabla u\|^2=\|\nabla\sigma_n\|^2+2\nabla\sigma_n\cdot\nabla v+\|\nabla v\|^2
        \end{align*}
        we see that
        \begin{align}
        \begin{split}\label{eq:gradient_u_and_v}
            \|\nabla u_n\|_2^2 = \|\nabla u\|_2^2 + \|\nabla \omega_n\|_2^2 + o(1),\ \|\nabla v_n\|_2^2 = \|\nabla v\|_2^2 + \|\nabla \sigma_n\|_2^2 +o(1).
        \end{split}
        \end{align}
        Since $u_n \rightharpoonup u$ and $v_n\rightharpoonup v$ in $H^1(\mathbb{R}^3)$, $\{\|u_n\|_{a_1, V_1}\}$ and $\{\|v_n\|_{a_2, V_2}\}$ are bounded sequence in $\mathbb{R}$. By the Sobolev embedding theorem, $\{u_n\}$ is bounded in $L^{\frac{6s}{3+\alpha}}(\mathbb{R}^3)$ for any $s \in [(3+\alpha)/3, 3+\alpha]$. Hence, using the nonlocal  Br\'{e}zis--Lieb Lemma(see Lemma \ref{lem:nonlocal-B-L}), we have
        \begin{align}
        \begin{split}\label{eq:Brezis-Lieb}
           & \int_{\mathbb{R}^3}(I_{\alpha}*|u_n|^s)|u_n|^sdx = \int_{\mathbb{R}^3}(I_{\alpha}*|u|^s)|u|^sdx + \int_{\mathbb{R}^3}(I_{\alpha}*|\omega_n|^s)|\omega_n|^sdx + o(1)\\
    & \int_{\mathbb{R}^3}(I_{\alpha}*|v_n|^s)|v_n|^sdx = \int_{\mathbb{R}^3}(I_{\alpha}*|v|^s)|v|^sdx + \int_{\mathbb{R}^3}(I_{\alpha}*|\sigma_n|^s)|\sigma_n|^sdx + o(1)\\
            &\quad\quad\quad \text{ for }\ \ s \in \left[\frac{3+\alpha}{3}, 3+\alpha\right].  
\end{split}
        \end{align}
        Combining \eqref{eq:gradient_u_and_v}, \eqref{eq:Brezis-Lieb}, we obtain that
        \begin{align}
            I_1(u_n, v_n) &= I_1(u, v) + I_1(\omega_n, \sigma_n) + o(1).     \label{eq:Identity of I_1} \\
            I_3(u_n, v_n) &= I_3(u, v) + I_3(\omega_n, \sigma_n) + o(1).     \label{eq:Identity of I_3}
        \end{align}
        Since
        \begin{align}\label{4thorder}
        \begin{split}
      \|\nabla u_n\|_2^4 &= (\|\nabla u\|_2^2 + \|\nabla \omega_n\|_2^2 + o(1))^2 \\
            &= \|\nabla u\|_2^4 + 2\|\nabla u\|_2^2\|\nabla \omega_n\|_2^2 + \|\nabla \omega_n\|_2^4 + o(1),\\
            \|\nabla v_n\|_2^4 &= (\|\nabla v\|_2^2 + \|\nabla \sigma_n\|_2^2 + o(1))^2 \\
            &= \|\nabla v\|_2^4 + 2\|\nabla v\|_2^2\|\nabla \sigma_n\|_2^2 + \|\nabla \sigma_n\|_2^4 + o(1),
            \end{split}
        \end{align}
        we have
        \begin{align}
            I_2(u_n, v_n) = I_2(u, v) + I_2(\omega_n, \sigma_n) + 2(b_1\|\nabla u\|_2^2\|\nabla \omega_n\|_2^2 + b_2\|\nabla v\|_2^2\|\nabla \sigma_n\|_2^2) + o(1).    \label{eq:Identity of I_2}
        \end{align}
        For the coupling term, we write
        \begin{align*}
            \int_{\mathbb{R}^3} u_nv_n\,dx = \int_{\mathbb{R}^3} \omega_n\sigma_n\,dx + \int_{\mathbb{R}^3} \omega_nv\,dx + \int_{\mathbb{R}^3} u\sigma_n\,dx + \int_{\mathbb{R}^3} uv\,dx.
        \end{align*}
        We claim that 
        \begin{align*}
        \int_{\mathbb{R}^3} \omega_nv\,dx \to 0\ \ \text{and}\ \ \int_{\mathbb{R}^3} u\sigma_n\,dx \to 0\ \ \text{as}\ \ n \to \infty. 
        \end{align*}
        Let
        \begin{align*}
            f(\omega) \coloneqq \int_{\mathbb{R}^3} \omega v\,dx, \quad \omega \in H^1(\mathbb{R}^3)
        \end{align*}
        By the Schwartz inequality, we see that
        \begin{align*}
            |f(\omega)| &\le \|\omega\|_2\|v\|_2 \\
            &\le \frac{1}{\sqrt{V_1}}\|v\|_2\|\omega\|_{a_1, V_1}.
        \end{align*}
        Hence $f$ is a bounded linear functional on $H^1(\mathbb{R}^3)$. Therefore $f(\omega_n) \to 0$ since $\omega_n \rightharpoonup 0$. Similarly, we get 
        \begin{align*}
            \int_{\mathbb{R}^3} u\sigma_n\,dx \to 0.
        \end{align*}
     Thus, we obtain that
\begin{align}
\label{product}
\int_{\mathbb{R}^3}u_nv_ndx=\int_{\mathbb{R}^3}uvdx+\int_{\mathbb{R}^3}\omega_n\sigma_ndx+o(1),
\end{align}
        that is,
        \begin{align}
            I_4(u_n, v_n) &= I_4(u, v) + I_4(\omega_n, \sigma_n) + o(1).     \label{eq:Identity of I_4}
        \end{align}
        Combining \eqref{eq:Identity of I_1}, \eqref{eq:Identity of I_3}, \eqref{eq:Identity of I_2} with \eqref{eq:Identity of I_4}, we obtain \eqref{eq:Identity of I}. Similarly we can obtain \eqref{eq:Identity of I'} and \eqref{eq:Identity of J}. 
    \end{proof}

    \section{Proof of Theorem \ref{Th:non-critical} }

 In this section we prove Theorem \ref{Th:non-critical}. 
  \begin{Lem}     \label{Lem:m is achieved}
        Assume that \ref{assumption:potential} holds. Then $m$ is achieved for some $(u,v)\in\mathcal{M}$, that is, there exists $(u,v)\in\mathcal{M}$ such that $I(u,v)=m$.
    \end{Lem}
    \begin{proof}
        We define the $C^1$ functional $\Phi \colon H \rightarrow \mathbb{R}$ by
        \begin{align} \label{eq:functional Psi} 
            \begin{split}
                \Phi(u, v) &\coloneqq I(u,v)-\frac{1}{8}J(u,v) \\
     &=\frac{1}{4}(a_1\|\nabla u\|_2^2+a_2\|\nabla v\|_2^2) \\
     &\ \ \ \ \ \ \ \ \ \ +\frac{p+\alpha-1}{8p}\mu\int_{\mathbb{R}^3}(I_{\alpha}*|u|^p)|u|^pdx+\frac{q+\alpha-1}{8q}\nu\int_{\mathbb{R}^3}(I_{\alpha}*|v|^q)|v|^qdx
            \end{split}
        \end{align}
        for $(u, v) \in H$. We note that $\Phi(u, v) = I(u, v) \ge m$ for $(u, v) \in \mathcal{M}$. Let $\{(u_n, v_n)\} \subset \mathcal{M}$ be a minimizing sequence for $m$, that is,
        \begin{align}
            I(u_n, v_n) \to m=\inf_{(u, v) \in \mathcal{M}} I(u, v), \quad J(u_n, v_n) = 0.   \label{eq:minimizing seq}
        \end{align}
        By \eqref{eq:I is positive}, we can see that $\{(u_n, v_n)\}$ is bounded in $H$. Let us show that there exist $\xi, R > 0$ such that
        \begin{align} \label{eq:concentration compactness}
            \sup_{y \in \mathbb{R}^3} \int_{B_R(y)} (u_n^2 + v_n^2)\,dx > \xi.  
        \end{align}
        holds. In fact, suppose by contradiction that \eqref{eq:concentration compactness} does not hold. Then, for any $R > 0$ we have
        \begin{align*}
            \lim_{n \to \infty} \sup_{y \in \mathbb{R}^3} \int_{B_R(y)} (u_n^2 + v_n^2)\,dx = 0.
        \end{align*}
        By Lemma \ref{Lions-th} we see that $u_n \to 0$, $v_n \to 0$ in $L^s(\mathbb{R}^3)$, for any $s\in (2,6)$. Since $(u_n, v_n) \in \mathcal{M}$, using \eqref{eq:estimate-H1-on-M}, Lemma \ref{Lem:positive on NP} and the Hardy--Littlewood--Sobolev inequality(Lemma \ref{H-L-S-ineq} and \eqref{eq:convolution-term-est}), we get 
        \begin{align}\label{eq:pr_of_Thm A}
            0 <2\left(1 - \delta\right)\underline{C}^2 &\le \frac{3}{2}\left(1 - \delta\right)\|(u_n, v_n)\|^2 \\
            &\le \frac{p + \alpha+3}{p}\mu\int_{\mathbb{R}^3}(I_{\alpha}*|u_n|^p)|u_n|^pdx + \frac{q +\alpha+ 3}{q}\nu\int_{\mathbb{R}^3}(I_{\alpha}*|v_n|^q)|v_n|^qdx \\
            &\le \frac{p+\alpha+3}{p}\mu C_{\mathrm{HLS}}(3,\alpha, {\textstyle\frac{6p}{3+\alpha}})\|u_n\|_{\frac{6p}{3+\alpha}}^{2p}+\frac{q+\alpha+3}{q}\nu C_{\mathrm{HLS}}(3,\alpha,{\textstyle\frac{6q}{3+\alpha}})\|v_n\|_{\frac{6q}{3+\alpha}}^{2q}\\
            &\to 0 \quad (n \to \infty),
        \end{align}
        which leads to contradiction. Therefore, \eqref{eq:concentration compactness} holds and there exists $\{y_n\} \subset \mathbb{R}^3$ such that
        \begin{align}\label{eq:concentration_compactness_2}
            \int_{B_R(y_n)} (u_n^2 + v_n^2)\,dx > \xi.
        \end{align}
        Let $\widetilde{u}_n(x) \coloneqq u_n(x + y_n)$ and $\widetilde{v}_n(x) \coloneqq v_n(x + y_n)$. By \eqref{eq:concentration_compactness_2} we have
                \begin{align*}
            \int_{B_R(0)} (\widetilde{u}_n^2 + \widetilde{v}_n^2)\,dx > \xi.
        \end{align*}
        Since 
        \begin{align*}
       \|\widetilde{u}_n\|_2^2 = \|u_n\|_2^2,\ \|\widetilde{v}_n\|_2^2 = \|v_n\|_2^2,\ 
       \|\nabla\widetilde{u}_n\|_2^2=\|\nabla u\|_2^2,\ \|\nabla\widetilde{v}_n\|_2^2=\|\nabla v_n\|_2^2
       \end{align*}
       and 
       \begin{align*}
       &\int_{\mathbb{R}^3}(I_{\alpha}*|\widetilde{u}_n|^p)|\widetilde{u}_n|^pdx=\int_{\mathbb{R}^3}(I_{\alpha}*|u_n|^p)|u_n|^pdx,\\
       &\int_{\mathbb{R}^3}(I_{\alpha}*|\widetilde{v}_n|^q)|\widetilde{v}_n|^qdx=\int_{\mathbb{R}^3}(I_{\alpha}*|v_n|^p)|v_n|^pdx
       \end{align*}
       we have
        \begin{align}
            I(\widetilde{u}_n, \widetilde{v}_n) \to m, \quad J(\widetilde{u}_n, \widetilde{v}_n) = 0.   \label{eq:it is minimizing seq}
        \end{align}
        Therefore, passing to a subsequence if necessary, we may assume that there exists $(u, v) \in H \setminus \{(0, 0)\}$ such that
        \begin{empheq}[left = \empheqlbrace]{align}
            \begin{split}
                & (\widetilde{u}_n, \widetilde{v}_n) \rightharpoonup (u, v) \text{ weakly in } H,   \label{eq:w-limit} \\
                & (\widetilde{u}_n,\widetilde{v}_n)\rightharpoonup (u,v)\ \text{weakly in}\ D^{1,2}(\mathbb{R}^3)\times D^{1,2}(\mathbb{R}^3), \\
                & (\widetilde{u}_n, \widetilde{v}_n) \to (u, v) \text{ strongly in } L_{\mathrm{loc}}^s(\mathbb{R}^3) \times L_{\mathrm{loc}}^s(\mathbb{R}^3) \text{ for all } s \in [2, 6) \\
                & (\widetilde{u}_n, \widetilde{v}_n) \to (u, v) \text{ almost everywhere in } \mathbb{R}^3.
            \end{split}
        \end{empheq}
        Let $\omega_n \coloneqq \widetilde{u}_n - u$ and $\sigma_n \coloneqq \widetilde{v}_n - v$. Then \eqref{eq:w-limit} and Lemma \ref{Lem:weak limit identity} yield
        \begin{align}\label{eq:Identity of Psi and weak limit}
            \Phi(\widetilde{u}_n, \widetilde{v}_n) = \Phi(u, v) + \Phi(\omega_n, \sigma_n) + o(1)   
        \end{align}
        and
        \begin{align}\label{eq:Identity of J and weak limit}
            J(\widetilde{u}_n, \widetilde{v}_n) = J(u, v) + J(\omega_n, \sigma_n) + 4(b_1\|\nabla u\|_2^2\|\nabla \omega_n\|_2^2 + b_2\|\nabla v\|_2^2\|\nabla \sigma_n\|_2^2) + o(1). 
        \end{align}
        By 
        \eqref{eq:it is minimizing seq}, \eqref{eq:Identity of Psi and weak limit} and \eqref{eq:Identity of J and weak limit}, we have
        \begin{align}\label{eq:Identity of Psi and J}
                \Phi(\omega_n, \sigma_n) = m - \Phi(u, v) + o(1), \quad J(\omega_n, \sigma_n) \le -J(u, v) + o(1). 
        \end{align}
        If there exist $\{\omega_{n_j}\} \subset \{\omega_n\}$ and $\{\sigma_{n_j}\} \subset \{\sigma_n\}$ such that $\omega_{n_j} = 0$ and $\sigma_{n_j} = 0$ for all $j \in \mathbb{N}$, then we heve
        \begin{align*}
                \Phi(u, v) = m\ \ \text{and}\ \ J(u, v) = 0,
        \end{align*}
        which implies the conclusion of Lemma \ref{Lem:m is achieved}. Next, we assume that $(\omega_n, \sigma_n) \neq (0, 0)$ for all large $n$. By Lemma \ref{lem:tuv_in_M} for each such $n$, there exists a unique $t_n > 0$ such that $((\omega_n)^{t_n}, (\sigma_n)^{t_n}) \in \mathcal{M}$. 
        
        Now we prove the following claim.
        \begin{Claim}
            $J(u, v) \le 0$.
        \end{Claim}
        If $J(u, v) > 0$, then \eqref{eq:Identity of Psi and J} implies $J(\omega_n, \sigma_n) < 0$ for large $n$. Using \eqref{eq:energy}, \eqref{eq:NP functional}, \eqref{eq:energy-ineq-02}, \eqref{eq:functional Psi} and \eqref{eq:Identity of Psi and J}, we obtain
        \begin{align}
            \begin{split}
                m - \Phi(u, v) + o(1) &= \Phi(\omega_n, \sigma_n) \label{eq:Identity of Psi and m}\\
                &= I(\omega_n, \sigma_n) - \frac{1}{8}J(\omega_n, \sigma_n) \\
                &\ge I((\omega_n)^{t_n}, (\sigma_n)^{t_n}) - \frac{t_n^8}{8}J(\omega_n, \sigma_n) \\
                &\ge m,
            \end{split}
        \end{align}
        which implies 
        \begin{align*}
                0\ge \Phi(u, v) \ge \frac{1}{4}(a_1\|\nabla u\|_2^2+a_2\|\nabla v\|_2^2)
        \end{align*}
        and $u=v=0$ almost everywhere in $\mathbb{R}^3$. This is a contradiction to $(u,v)\ne (0,0)$.  The claim has been proved. 
        
        From \eqref{eq:functional Psi} and \eqref{eq:it is minimizing seq}, we obtain the following estimate 
        \begin{align*}
            m &= \lim_{n \to \infty} \left[I(\widetilde{u}_n, \widetilde{v}_n) - \frac{1}{8}J(\widetilde{u}_n, \widetilde{v}_n)\right] \\
            &= \lim_{n \to \infty} \left[\frac{1}{4}(a_1\|\nabla \widetilde{u}_n\|_2^2 + a_2\|\nabla \widetilde{v}_n\|_2^2)  \right.\\
            &\quad \left. + \frac{p+\alpha - 1}{8p}\mu \int_{\mathbb{R}^3}(I_{\alpha}*|\widetilde{u}_n|^p)|\widetilde{u}_n|^pdx + \frac{q +\alpha- 1}{8q}\nu\int_{\mathbb{R}^3}(I_{\alpha}*|\widetilde{v}_n|^q)|\widetilde{v}_n|_q^qdx \right] 
        \end{align*}
        By the weak lower semicontinuity of norms of $D^{1,2}(\mathbb{R}^3)$ and Lemma \ref{lem:nonlocal-B-L}, we have
        \begin{align}\label{ineq:m>I-J/8}
        \begin{split}
            m &\ge \frac{1}{4}(a_1\|\nabla u\|_2^2 + a_2\|\nabla v\|_2^2) \\
            &\quad + \frac{p +\alpha- 1}{8p}\mu\int_{\mathbb{R}^3}(I_{\alpha}*|u|^p)|u|^pdx + \frac{q +\alpha- 1}{8q}\nu\int_{\mathbb{R}^3}(I_{\alpha}*|v|^q)|v|^qdx  \\
            &= I(u, v) - \frac{1}{8}J(u, v). 
\end{split}
    \end{align}
    On the other hand, since $(u, v) \neq (0, 0)$, there exists $t > 0$ such that $(u^t, v^t) \in \mathcal{M}$. Then by Lemma \ref{eq:energy-ineq-02} we get
    \begin{align}\label{ineq:I-J/8>m}
    I(u,v)-\frac{1}{8}J(u,v)\ge I(u^t, v^t) - \frac{t^8}{8}J(u, v) \ge m.
    \end{align}
    Combining \eqref{ineq:m>I-J/8} and \eqref{ineq:I-J/8>m} we can obtain
    \begin{align}\label{m=I-J/8}
   I(u,v)-\frac{1}{8}J(u,v)=I(u^t, v^t)-\frac{t^8}{8}J(u,v)=m.
    \end{align}
    From the above identities we can conclude that 
        \begin{align*}
            J(u, v) = 0\ \ \text{and}\ \ I(u, v) = m
        \end{align*}
holds. In fact if $J(u,v)<0$, then $I(u^t,v^t)=m+(t^8/8)J(u,v)<m$, which is contradiction to $(u^t, v^t)\in\mathcal{M}$ and the definition of $m$. Therefore, we conclude that $(u,v)\in \mathcal{M}$ and that $m$ is indeed attained at $(u,v)$.
\end{proof}

In the next lemma we prove the minimizer is actually a critical point of $I$ on $H$. We use the deformation lemma and degree theory.
   \begin{Lem}     \label{Lem:ground state}
        Assume that \ref{assumption:potential} holds. If $(u, v) \in \mathcal{M}$ satisfies $m = I(u, v)$ with $m$ defined in \eqref{eq:Energy level}, then $(u, v)$ is a critical point of $I$.
    \end{Lem}
    \begin{proof} Although this lemma can be proved by the same argument as in \cite{Tatsuya} we give the proof for reader's convenience. 
    
        Assume that $I^{\prime}(u, v) \neq 0$. Then there exist $\tau > 0$ and $\rho > 0$ such that
        \begin{align*}
            \|(\widetilde{u}, \widetilde{v}) - (u, v)\| \le 3\tau\ \ \Rightarrow\ \  \|I^{\prime}(\widetilde{u}, \widetilde{v})\|_{H^*} \ge \rho.
        \end{align*}
        We first claim that
        \begin{align}\label{eq:limit for t}
            \lim_{t \to 1} \|(u^t, v^t) - (u, v)\| = 0.   
        \end{align}
        Arguing by contradiction, suppose that there exist $\varepsilon_0 > 0$ and $\{t_n\} \subset \mathbb{R}$ such that
        \begin{align}
            \lim_{n \to \infty} t_n = 1\ \ \ \text{and}\ \ \ \|(u^{t_n}, v^{t_n}) - (u, v)\| \ge \varepsilon_0.     \label{eq:Identity for t_n}
        \end{align}
        Since $C_0^{\infty}(\mathbb{R}^3)$ is dense in $L^2(\mathbb{R}^3)$ and $C_0^{\infty}(\mathbb{R}^3) \times C_0^{\infty}(\mathbb{R}^3)$ is dense in $H$, there exist $U_1, U_2 \in C_0(\mathbb{R}^3, \mathbb{R}^3)$ and $\varphi_1, \varphi_2 \in C_0(\mathbb{R}^3)$ such that
        \begin{align}
            \begin{split}
                & a_1\|\nabla u - U_1\|_2^2 < \frac{\varepsilon_0}{60}, \quad a_2\|\nabla v - U_2\|_2^2 < \frac{\varepsilon_0}{60},    \label{eq:approximation} \\
                & V_1\|u - \varphi_1\|_2^2 < \frac{\varepsilon_0}{40}, \quad V_2\|v - \varphi_2\|_2^2 < \frac{\varepsilon_0}{40}.
            \end{split}
        \end{align}
        By \eqref{eq:Identity for t_n} and \eqref{eq:approximation}, we have
        \begin{align}
            \begin{split}
                a_1\|\nabla (u^{t_n}) - \nabla u\|_2^2 &= a_1\int_{\mathbb{R}^3} |\nabla (u^{t_n}) - \nabla u|^2\,dx     \label{eq:L^2 norm of gradient} \\
                &\le 2a_1\left(\int_{\mathbb{R}^3} \left|\nabla u\left(\frac{x}{t_n}\right) - U_1\right|^2\,dx + \int_{\mathbb{R}^3} |U_1 - \nabla u|^2\,dx\right) \\
                &\le 4a_1\left(\int_{\mathbb{R}^3} \left|\nabla u\left(\frac{x}{t_n}\right) - U_1\left(\frac{x}{t_n}\right)\right|^2\,dx + \int_{\mathbb{R}^3} \left|U_1\left(\frac{x}{t_n}\right) - U_1(x)\right|^2\,dx \right) \\
                &\quad + 2a_1\int_{\mathbb{R}^3} |U_1 - \nabla u|^2\,dx \\
                &= 4a_1t_n^3\int_{\mathbb{R}^3} |\nabla u(y) - U_1(y)|^2\,dy + 4a_1\int_{\mathbb{R}^3} \left|U_1\left(\frac{x}{t_n}\right) - U_1(x)\right|^2\,dx \\
                &\quad + 2a_1\int_{\mathbb{R}^3} |U_1 - \nabla u|^2\,dx \\
                &\le \frac{a_1(1 + 2t_n^3)}{30}\varepsilon_0 + 4\int_{\mathbb{R}^3} \left|U_1\left(\frac{x}{t_n}\right) - U_1(x)\right|^2\,dx \\
                &= \frac{1}{10}\varepsilon_0 + o(1),
            \end{split}
        \end{align}
        and
        \begin{align}
            \begin{split}
                V_1\|u^{t_n} - u\|_2^2 &= V_1\int_{\mathbb{R}^3} |u^{t_n} - u|^2\,dx      \label{eq:L^2 norm of u}\\
                &\le 2V_1\left(\int_{\mathbb{R}^3} \left|t_nu\left(\frac{x}{t_n}\right) - \varphi_1\right|^2\,dx + \int_{\mathbb{R}^3} |\varphi_1 - u|^2\,dx\right) \\
                &\le 6V_1\left(\int_{\mathbb{R}^3} \left|t_nu\left(\frac{x}{t_n}\right) - t_n\varphi_1\left(\frac{x}{t_n}\right)\right|^2\,dx + \int_{\mathbb{R}^3} \left|t_n\varphi_1\left(\frac{x}{t_n}\right) - t_n\varphi_1(x)\right|^2\,dx \right.\\
                &\quad \left. + \int_{\mathbb{R}^3} |t_n\varphi_1(x) - \varphi_1(x)|^2\,dx\right) + 2V_1\int_{\mathbb{R}^3} |\varphi_1 - u|^2\,dx \\
                &= 6V_1t_n^5\int_{\mathbb{R}^3} |u(y) - \varphi_1(y)|^2\,dy + 6V_1^2t_n^2\int_{\mathbb{R}^3} \left|\varphi_1\left(\frac{x}{t_n}\right) - \varphi_1(x)\right|^2\,dx \\
                &\quad + 6V_1|t_n - 1|^2\int_{\mathbb{R}^3} |\varphi_1|^2\,dx + 2V_1\int_{\mathbb{R}^3} |\varphi_1 - u|^2\,dx \\
                &\le \frac{V_1(1 + 3t_n^5)}{20}\varepsilon_0 + 6V_1t_n^2\int_{\mathbb{R}^3} \left|\varphi_1\left(\frac{x}{t_n}\right) - \varphi_1(x)\right|^2\,dx + 6V_1|t_n - 1|^2\int_{\mathbb{R}^3} |\varphi_1|^2\,dx \\
                &= \frac{1}{5}\varepsilon_0 + o(1).
            \end{split}
        \end{align}
        Combining \eqref{eq:L^2 norm of gradient} with \eqref{eq:L^2 norm of u}, we have
        \begin{align*}
            \|u^{t_n} - u\|_{a_1, V_1}^2 &= a_1|\nabla (u^{t_n}) - \nabla u|_2^2 + V_1|u^{t_n} - u|_2^2 \\
            &\le \frac{3}{10}\varepsilon_0 + o(1).
        \end{align*}
        Similarly, we get $\|v^{t_n} - v\|_{a_2, V_2}^2 \le \frac{3}{10}\varepsilon_0 + o(1)$. Hence we obtain
        \begin{align}
            \|(u^{t_n}, v^{t_n}) - (u, v)\| \le \frac{3}{5}\varepsilon_0 + o(1).  \label{eq:Identity for t_n is less than varepsilon_0}
        \end{align}
        But \eqref{eq:Identity for t_n is less than varepsilon_0} contradicts with \eqref{eq:Identity for t_n}. Therefore, \eqref{eq:limit for t} holds. Thus, there exist $\tau_1 > 0$ such that
        \begin{align}
            |t - 1| < \tau_1\ \ \Rightarrow\ \  \|(u^t, v^t) - (u, v)\| < \tau.  \label{eq:eq:(t_nu_(t_n), t_nv_(t_n)) - (u, v) is bounded}
        \end{align}
        By \eqref{eq:energy-ineq-02}, we have
        \begin{align}\label{eq:Evaluate I(tu_t, tv_t) from above}
            \begin{split}
                I(u^t, v^t) &\le I(u, v) - \frac{(1 - t^4)^2}{4}(a_1\|\nabla u\|_2^2 + a_2\|\nabla v\|_2^2)    \\
                &= m - \frac{(1 - t^4)^2}{4}(a_1\|\nabla u\|_2^2 + a_2\|\nabla v\|_2^2), \quad t > 0.
            \end{split}
        \end{align}
        By \eqref{eq:Evaluate I(tu_t, tv_t) from above} and
        \begin{align*}
        (1-t^4)^2=(1-t)^2(1+t)^2(1+t^2)^2\ge (1-t)^2\ \ \text{for}\ \ t>0
        \end{align*}
        we have
        \begin{align}\label{eq:Evaluate I(tu_t, tv_t) from above2}
  I(u^t, v^t)\le m-\frac{(1-t)^2}{4}(a_1\|\nabla u\|_2^2 + a_2\|\nabla v\|_2^2)\ \ \text{for}\ \ t>0.
        \end{align}
        Let $\varepsilon \coloneqq \min \{(a_1\|\nabla u\|_2^2 + a_2\|\nabla v\|_2^2)/32, 1, \rho\tau/8\}$ and $\mathcal{B}_{\tau}((u, v))$, the ball in $H$ with radius $\tau$ and centerd at $(u,v)$. Then \cite[Lemma 2.3]{Willem} yields a deformation $\eta \in C([0, 1] \times H, H)$ such that
        \begin{enumerate}
            \renewcommand{\labelenumi}{(\roman{enumi})}
            \item $\eta(s, (u, v)) = (u, v)$ if $s = 0$ or $I(u, v) < m - 2\varepsilon$ or $I(u, v) > m + 2\varepsilon$;
            \item $\eta(1, I^{m + \varepsilon} \cap \mathcal{B}_{\tau}((u, v))) \subset I^{m - \varepsilon}$;
            \item $I(\eta(1, (\varphi, \psi))) \le I(\varphi, \psi) \quad \text{for}\ (\varphi, \psi) \in H$;
            \item $\eta(1, (u, v))$ is a homeomorphism of $H$,
        \end{enumerate}
        where $I^c = \{(u, v) \in H \mid I(u, v) \le c\}$ for $c \in \mathbb{R}$. By Corollary \ref{Cor:maximum}, $I(u^t, v^t) \le I(u, v) = m$ for $t \ge 0$, then it follows from \eqref{eq:eq:(t_nu_(t_n), t_nv_(t_n)) - (u, v) is bounded} and (ii) that
        \begin{align}
            I(\eta(1, (u^t, v^t))) \le m - \varepsilon \text{ for } t \ge 0 \text{ with } |t - 1| < \tau_1   \label{eq:I(eta) is less than m - epsilon}.
        \end{align}
        On the other hand, by (iii), \eqref{eq:Evaluate I(tu_t, tv_t) from above2} and 
        we have
    \begin{align}\label{eq:Evaluate I(eta) from above}
            \begin{split}
                I(\eta(1, (u^t, v^t))) &\le I(u^t, v^t)     \\
                &\le m - \frac{\tau_1^2}{4}(a_1\|\nabla u\|_2^2 + a_2\|\nabla v\|_2^2), \quad \text{for}\ \ |t - 1| \ge \tau_1\ \ \text{with}\ \ t\ge 0.
            \end{split}
        \end{align}
        Combining \eqref{eq:I(eta) is less than m - epsilon} and \eqref{eq:Evaluate I(eta) from above}, we have
        \begin{align}
            \max_{t \in [0.5, 1.5]} I(\eta(1, (u^t, v^t))) < m.   \label{eq:I < m}
        \end{align}
        We claim that there exists $t \in [0.5, 1.5]$ such that $\eta(1, (tu_t, tv_t)) \cap \mathcal{M} \neq \emptyset$. Define
        \begin{align*}
            \Psi_0(t) \coloneqq J(u^t, v^t), \quad \Psi_1(t) \coloneqq J(\eta(1, (u^t, v^t)))\ \text{ for }\ t > 0.
        \end{align*}
        By a direct computation we have
        \begin{align*}
            \Psi_0(t)&=J(u^t, v^t) \\
                     &=2t^4(a_1\|\nabla u\|_2^2 + a_2\|\nabla v\|_2^2) + 4t^8\left(V_1\|u\|_2^2 + V_2\|v\|_2^2\right)+ 2t^8(b_1\|\nabla u\|_2^4 + b_2\|\nabla v\|_2^4) \\
            &\quad - \frac{p + \alpha+3}{p}\mu t^{2(p+\alpha+3)}\int_{\mathbb{R}^3}(I_{\alpha}*|u|^p)|u|^pdx - \frac{q +\alpha+ 3}{q}\nu t^{2(q+\alpha+3)}\int_{\mathbb{R}^3}(I_{\alpha}*|v|^q)|v|^qdx  \\
            &\quad- 8\lambda t^8\int_{\mathbb{R}^3} uv\,dx.
        \end{align*}
        Hence we obtain
        \begin{align*}
            \Psi_0'(t) &=8t^3(a_1\|\nabla u\|_2^2 + a_2\|\nabla v\|_2^2) + 32t^7\left(V_1\|u\|_2^2 + V_2\|v\|_2^2\right)+ 16t^7(b_1\|\nabla u\|_2^4 + b_2\|\nabla v\|_2^4) \\
            &\quad - \frac{2(p + \alpha+3)^2}{p}\mu t^{2(p+\alpha+3)-1}\int_{\mathbb{R}^3}(I_{\alpha}*|u|^p)|u|^pdx \\
            &\quad - \frac{2(q +\alpha+ 3)^2}{q}\nu t^{2(q+\alpha+3)-1}\int_{\mathbb{R}^3}(I_{\alpha}*|v|^q)|v|^qdx- 64\lambda t^7\int_{\mathbb{R}^3} uv\,dx.
        \end{align*}
        By using $\Phi_0(1) = J(u,v) = 0$, we have
        \begin{align*}
            \Psi_0'(1)&=\Psi_0'(1)-8\Psi_0(1) \\
            &=-8(a_1\|\nabla u\|_2^2 + a_2\|\nabla v\|_2^2)-\frac{(p+\alpha+3)(p+\alpha-1)}{p}\mu\int_{\mathbb{R}^3}(I_{\alpha}*|u|^p)|u|^pdx \\
            &\quad \quad -\frac{(q+\alpha+3)(q+\alpha-1)}{q}\nu\int_{\mathbb{R}^3}(I_{\alpha}*|v|^q)|v|^qdx<0.        \end{align*}

        Since $(u, v) \in \mathcal{M}, \Psi_0(t) = 0$ holds only if $t = 1$ by \eqref{eq:NPmanifold}. We see that $\deg(\Psi_0, (0.5, 1.5), 0) = -1$. If $t = 0.5$ or $t = 1.5$, then it follows from \eqref{eq:Evaluate I(tu_t, tv_t) from above2}, (i) and the choice of $\varepsilon$ that
        \begin{align*}
            \eta(1, (u^t, v^t)) = (u^t, v^t)\ \ \text{ for }\ \ t = 0.5 \text{ or } t = 1.5.
        \end{align*}
        Noting that
        \begin{align*}
            & \Psi_0(t) = J(u^t, v^t) = J(\eta(0, (u^t, v^t))), \\
            & \Psi_1(t) = J(\eta(1, (u^t, v^t))),
        \end{align*}
        we define
        \begin{align*}
            \Psi(t, s) \coloneqq J(\eta(s, (u^t, v^t))), \quad (t, s) \in [0.5, 1.5] \times [0, 1].
        \end{align*}
        Then we can see that $\Psi(t, s)$ is a homotopy between $\Psi_0(t)$ and $\Psi_1(t)$. Hence, by the fundamental property the degree (see for example \cite{degree theory}), we see that
        \begin{align*}
            \deg(\Psi_1, (0.5, 1.5), 0) = \deg(\Psi_0, (0.5, 1.5), 0) = -1.
        \end{align*}
        This means that there exists $\bar{t} \in (0.5, 1.5)$ such that $\Psi_1(\bar{t}) = 0$. Therefore we obtain $\eta(1, (u^{\bar{t}}, v^{\bar{t}})) \in \mathcal{M}$. This and \eqref{eq:I < m} contradict the definition of $m$.
    \end{proof}
    Now we are in the position to prove Theorem \ref{Th:non-critical}. 
    \begin{proof}[Proof of Theorem \ref{Th:non-critical}]
        In the preceding lemma we have obtained a ground state solution $(u, v) \in H$ to \eqref{eq:NKC}. In order to get a nonnegative ground state solution, using Lemma \ref{lem:tuv_in_M}, there exists $t_0 > 0$ such that $(|u|^{t_0}, |v|^{t_0}) \in \mathcal{M}$. Thus, we see that
        \begin{align*}
            I(|u|^{t_0}, |v|^{t_0}) \le I(u^{t_0}, v^{t_0}) \le I(u, v) = m,
        \end{align*}
        which means that $(|u|^{t_0}, |v|^{t_0})$ is also a minimizer of $m$ and $(|u|^{t_0}, |v|^{t_0})$ is a nonnegative ground state solution to \eqref{eq:NKC}. Since $(|u|^{t_0}, |v|^{t_0}) \neq (0, 0)$, we may assume without loss of general that $|u|^{t_0} \neq 0$. We claim that $|v|^{t_0} \neq 0$. In fact, since $(|u|^{t_0}, |v|^{t_0})$ is a critical point of $I$, if $|v|^{t_0} = 0$, then for any $\varphi \in C_0^{\infty}(\mathbb{R}^3)$  we have
        \begin{align*}
            0 = \langle I^{\prime}(|u|^{t_0}, |v|^{t_0}), (0, \varphi^{t_0}) \rangle = -\lambda t_0^8 \int_{\mathbb{R}^3} |u|\varphi\,dx.
        \end{align*}
        Since $\lambda, t_0 > 0$, we have $|u| = 0$, which is a contradiction. Hence $|v|^{t_0} \neq 0$. By the strong maximum principle in each equation of \eqref{eq:NKC}, we see that $(|u|^{t_0}, |v|^{t_0})$ is positive.
        The proof of Theorem \ref{Th:non-critical} has been completed.
    \end{proof}

    \section{Proof of Theorem \ref{Th:critical}}
  In this section, we are concerned with the critical cases, that is, the upper half critical case $3 < p < q = 3+\alpha$ and the lower half critical case $1+\frac{\alpha}{3}=p<q<3+\alpha$. 

  To stress the dependence of $m$ on $\mu$ and $\nu$ we denote 
  \begin{align*}
  m_{\mu, \nu}=\inf_{(u,v)\in\mathcal{M}}I(u,v).
  \end{align*}
We begin with a lemma about asymptotic behaviors of $m_{\mu, \nu}$ as $\mu\to\infty$ or $\nu\to\infty$.
\begin{Lem}\label{lem:critical-lemma}
Suppose that $\frac{3 + \alpha}{3} \leq p \leq q \leq 3 + \alpha$. Then the following statements hold:
\begin{enumerate}
\item[\textup{(1)}] For each fixed $\nu$, we have $\lim_{\mu \to \infty} m_{\mu, \nu} = 0$.
\item[\textup{(2)}] For each fixed $\mu$, we have $\lim_{\nu \to \infty} m_{\mu, \nu} = 0$.
\end{enumerate}
\end{Lem}

\begin{proof}
We only prove (1), since (2) can be proved in a similar manner.  
Take any $\phi \in C_0^{\infty}(\mathbb{R}^3)$ such that $\phi \geq 0$ and $\phi \not\equiv 0$.  
By Lemma~\ref{lem:tuv_in_M}, there exists a unique $t_{\mu} > 0$ such that $(\phi^{t_\mu}, \phi^{t_\mu}) \in \mathcal{M}$.  
We claim that $t_\mu \to 0$ as $\mu \to \infty$.

Since $J(\phi^{t_\mu}, \phi^{t_\mu}) = 0$, it follows that
\begin{align}\label{eq:tmu_to_0-1}
\begin{split}
& 2t_\mu^4(a_1\|\nabla \phi\|_2^2 + a_2\|\nabla \phi\|_2^2) 
+ 4t_\mu^8(V_1\|\phi\|_2^2 + V_2\|\phi\|_2^2) 
+ 2t_\mu^8(b_1\|\nabla \phi\|_2^4 + b_2\|\nabla \phi\|_2^2) \\
&\quad - \frac{p + \alpha + 3}{p}\, \mu\, t_\mu^{2(p + \alpha + 3)} \int_{\mathbb{R}^3}(I_{\alpha} * |\phi|^p)|\phi|^p \, dx \\
&\quad - \frac{q + \alpha + 3}{q}\, \nu\, t_\mu^{2(q + \alpha + 3)} \int_{\mathbb{R}^3}(I_{\alpha} * |\phi|^q)|\phi|^q \, dx 
- 8\lambda t_\mu^8 \int_{\mathbb{R}^3} \phi^2 \, dx = 0.
\end{split}
\end{align}

From this identity we deduce:
\begin{align*}
& 2t_\mu^{-4}(a_1\|\nabla \phi\|_2^2 + a_2\|\nabla \phi\|_2^2) 
+ 4\left(V_1\|\phi\|_2^2 + V_2\|\phi\|_2^2- 2\lambda \int_{\mathbb{R}^3} \phi^2 \, dx\right) \\
&\quad + 2(b_1\|\nabla \phi\|_2^4 + b_2\|\nabla \phi\|_2^4) \\
&\geq \frac{q + \alpha + 3}{q}\, \nu\, t_\mu^{2(q + \alpha - 1)} \int_{\mathbb{R}^3}(I_{\alpha} * |\phi|^q)|\phi|^q \, dx.
\end{align*}

This shows that the family $\{t_\mu\}_\mu$ of positive numbers is bounded.  
Moreover, from \eqref{eq:tmu_to_0-1} we obtain
\begin{align}\label{eq:tmu_to_0-2}
\begin{split}
\frac{1}{\mu} \bigg\{ 
& 2t_\mu^{-4}(a_1\|\nabla \phi\|_2^2 + a_2\|\nabla \phi\|_2^2) 
+ 4(V_1\|\phi\|_2^2 + V_2\|\phi\|_2^2) 
- 8\lambda \int_{\mathbb{R}^3} \phi^2 \, dx \\
& + 2(b_1\|\nabla \phi\|_2^4 + b_2\|\nabla \phi\|_2^4) \bigg\}
\geq \frac{p + \alpha + 3}{p} t_\mu^{2(p + \alpha - 1)} \int_{\mathbb{R}^3}(I_{\alpha} * |\phi|^p)|\phi|^p \, dx.
\end{split}
\end{align}

Letting $\mu \to \infty$ in this inequality, we deduce $t_\mu \to 0$ and hence $\mu t_\mu^{2(p + \alpha + 3)} \to 0$.  

Now, for each $\mu > 0$, we have
\begin{align*}
m_{\mu, \nu} &\leq I(\phi^{t_\mu}, \phi^{t_\mu}) \\
&= \frac{1}{2} t_\mu^4 (a_1\|\nabla \phi\|_2^2 + a_2\|\nabla \phi\|_2^2) 
+ \frac{1}{2} t_\mu^8 (V_1\|\phi\|_2^2 + V_2\|\phi\|_2^2) 
+ \frac{1}{4} t_\mu^8 (b_1\|\nabla \phi\|_2^4 + b_2\|\nabla \phi\|_2^2) \\
&\quad - \frac{1}{2p}(\mu t_\mu^{2(p + \alpha + 3)}) \int_{\mathbb{R}^3}(I_{\alpha} * |\phi|^p)|\phi|^p \, dx 
- \frac{\nu}{2q} t_\mu^{2(q + \alpha + 3)} \int_{\mathbb{R}^3}(I_{\alpha} * |\phi|^q)|\phi|^q \, dx \\
&\quad - \lambda t_\mu^8 \int_{\mathbb{R}^3} \phi^2 \, dx \to 0
\end{align*}
as $\mu \to \infty$. This completes the proof.
\end{proof}

\subsection{The Upper half critical Case $\bm{1+\alpha/3<p<q=3+\alpha}$}

In this subsection, we prove (1) of Theorem \ref{Th:critical}. By Lemma \ref{lem:critical-lemma} for each fixed $\nu> 0$ there exists $\mu_0=\mu_0(\nu)>0$ such that 
\begin{align}\label{eq:mu_big}
m_{\mu,\nu}<\frac{\alpha+1}{4(3+\alpha)}\nu\left(\frac{a_2}{\nu}\mathcal{S}^*\right)^{\frac{3+\alpha}{2+\alpha}}\ \ \mbox{for}\ \mu\ge\mu_0.
\end{align}
Fix $\mu\ge \mu_0$ and take a minimizing sequence $\{(u_n, v_n)\} \subset \mathcal{M}$ for $m_{\mu,\nu}$. By \eqref{eq:I is positive}, $\{(u_n, v_n)\}$ is bounded in $H$ and we may assume that $(u_n, v_n) \rightharpoonup (u, v)$ weakly in $H$.

By the Hardy--Littlewood--Sobolev inequality, there exists $A(\mu) \in [0, \infty)$ such that \begin{align*}
\int_{\mathbb{R}^3}(I_{\alpha}*|v_n|^{3+\alpha})|v_n|^{3+\alpha} dx \to A(\mu)
 \end{align*}
along a subsequence. To complete the proof of (1) of Theorem \ref{Th:critical} it is enough to prove the following lemma.
\begin{Lem}  \label{Lem:compactness for critical}
        Assume that \ref{assumption:potential} holds and $1+\frac{\alpha}{3}<p<q=3+\alpha$. Let $\{(u_n, v_n)\}\subset\mathcal{M}$ be a minimizing sequence for $m$. If $\mu\ge\mu_0$, then there exist a constants $\xi > 0$ and $R > 0$ such that
        \begin{align}\label{eq:compactness for critical-1}
            \sup_{y \in \mathbb{R}^3} \int_{B_R(y)} (u_n^2 + v_n^2)\,dx > \xi.      
        \end{align}
    \end{Lem}
\begin{proof}
        Suppose by contradiction that \eqref{eq:compactness for critical-1} does not hold. Then, for any $R > 0$, we have
        \begin{align*}
            \lim_{n \to \infty} \sup_{y \in \mathbb{R}^3} \int_{B_R(y)} (u_n^2 + v_n^2)\,dx = 0.
        \end{align*}
        By Lemma \ref{Lions-th}, it follows that $u_n \to 0$ in $L^r(\mathbb{R}^3)$ for $2 < r < 6$. Since $\frac{3+\alpha}{3}<p<3+\alpha$ by the Hardy--Littlewood--Sobolev inequality we see that
        \begin{align*}
        \int_{\mathbb{R}^3}(I_{\alpha}*|u_n|^p)|u_n|^pdx\to 0.
        \end{align*}

        From \eqref{eq:NP functional} and the fact $J(u_n,v_n)=0$, it follows that 
        \begin{align}\label{eq:mu-big-2}
            \begin{split}
 &2(a_1\|\nabla u_n\|_2^2 + a_2\|\nabla v_n\|_2^2) + 4\left(V_1\|u_n\|_2^2 + V_2\|v_n\|_2^2- 2\lambda\int_{\mathbb{R}^3}u_nv_ndx\right) \\
 &\quad\quad +2(b_1\|\nabla u_n\|_2^4+b_2\|\nabla v_n\|_2^4)=2\nu A(\mu)+o(1)
            \end{split}
        \end{align}
Since the second and the third terms of the left-hand-side of the above identity is nonnegative, by \eqref{eq:upper_critical_constant} we see that 
\begin{align*}
2a_2\mathcal{S}^*A(\mu)^{\frac{1}{3+\alpha}}+o(1)&\le 2a_2\|\nabla v_n\|_2^2 \\
                                     &\le 2\nu A(\mu)+o(1).
\end{align*}
We now assume that $A(\mu)>0$. Then we obtain
\begin{align}\label{eq:A_mu_est}
\left(\frac{a_2}{\nu}\mathcal{S}^*\right)^{\frac{3+\alpha}{2+\alpha}}\le A(\mu). 
\end{align}
On the other hand, from \eqref{eq:A_mu_est} and \eqref{eq:functional Psi} we obtain
\begin{align*}
m_{\mu,\nu}+o(1)= &\ \frac{1}{2}(a_1\|\nabla u_n\|_2^2+a_2\|\nabla v_n\|_2^2) \\
     &\ \ +\frac{p+\alpha-1}{8p}\mu\int_{\mathbb{R}^3}(I_{\alpha}*|u_n|^p)|u_n|^pdx+\frac{2(\alpha+1)}{8(3+\alpha)}\nu\int_{\mathbb{R}^3}(I_{\alpha}*|v_n|^{3+\alpha})|v_n|^{3+\alpha}dx \\
     \ge&\ \frac{2(\alpha+1)}{8(\alpha+3)}\nu\left(\frac{a_2}{\nu}\mathcal{S}^*\right)^{\frac{3+\alpha}{2+\alpha}}+o(1). 
\end{align*}
and 
\begin{align*}
m_{\mu,\nu}\ge \frac{\alpha+1}{4(\alpha+3)}\nu\left(\frac{a_2}{\nu}\mathcal{S}^*\right)^{\frac{3+\alpha}{2+\alpha}}.
\end{align*}
This is a contradiction to \eqref{eq:mu_big}. Hence $A(\mu)=0$. However, by \eqref{eq:mu-big-2} this implies that $(u_n,v_n)\to 0$ in $H$ which is contradiction to Lemma \ref{Lem:positive on NP}.   
\end{proof}
\begin{proof}[Proof of (1) of Theorem B]
Since \eqref{eq:compactness for critical-1} holds, we can consider the shift sequence $(\Tilde{u}_n(x), \Tilde{v}_n(x)) \coloneqq (u_n(x + y_n), v_n(x + y_n))$ and we see that the weak limit is nontrivial. Therefore, we can repeat the same arguments as the proof of Theorem \ref{Th:non-critical}. The proof of 
(1) of Theorem \ref{Th:critical} has been completed.
\end{proof}

\subsection{The Lower half critical Case $\bm{1+\alpha/3=p<q<3+\alpha}$}

In this subsection, we prove (2) of Theorem \ref{Th:critical}. By Lemma \ref{lem:critical-lemma} for each fixed $\mu>0$ there exists $\nu_0=\nu_0(\mu)>0$ such that 
\begin{align}\label{eq:nu_big}
m_{\mu,\nu}<\frac{\alpha}{2(3+\alpha)}\mu\left(\frac{4(3+\alpha)}{\mu(12+\alpha)}V_1(1-\delta)\mathcal{S}_*\right)^{\frac{3+\alpha}{3}}\ \ \mbox{for}\ \nu\ge\nu_0.
\end{align}
Fix $\nu\ge \nu_0$ and take a minimizing sequence $\{(u_n, v_n)\} \subset \mathcal{M}$ for $m_{\mu,\nu}$. By \eqref{eq:I is positive}, $\{(u_n, v_n)\}$ is bounded in $H$ and we may assume that $(u_n, v_n) \rightharpoonup (u, v)$ weakly in $H$.

By the Hardy--Littlewood--Sobolev inequality, there exists $B(\nu) \in [0, \infty)$ such that \begin{align*}
\int_{\mathbb{R}^3}(I_{\alpha}*|v_n|^{1+\frac{\alpha}{3}})|v_n|^{1+\frac{\alpha}{3}} dx \to B(\nu) \in [0, +\infty)
 \end{align*}
along a subsequence. To complete the proof of (2) of Theorem \ref{Th:critical} it is enough to prove the following lemma.
\begin{Lem}  \label{Lem:compactness for critical-2}
        Assume that \ref{assumption:potential} holds and $p=1+\frac{\alpha}{3}<q<3+\alpha$. Let $\{(u_n, v_n)\}\subset\mathcal{M}$ be a minimizing sequence for $m$. If $\nu\ge\nu_0$, then there exist a constants $\xi > 0$ and $R > 0$ such that
        \begin{align}\label{eq:compactness for critical-2}
            \sup_{y \in \mathbb{R}^3} \int_{B_R(y)} (u_n^2 + v_n^2)\,dx > \xi.      
        \end{align}
    \end{Lem}
\begin{proof}
        Suppose by contradiction that \eqref{eq:compactness for critical-2} does not hold. Then, for any $R > 0$, we have
        \begin{align*}
            \lim_{n \to \infty} \sup_{y \in \mathbb{R}^3} \int_{B_R(y)} (u_n^2 + v_n^2)\,dx = 0.
        \end{align*}
        By Lemma \ref{Lions-th}, it follows that $u_n \to 0$ in $L^r(\mathbb{R}^3)$ for $2 < r < 6$. Since $\frac{3+\alpha}{3}<q<3+\alpha$ by the Hardy--Littlewood--Sobolev inequality we see that
        \begin{align*}
        \int_{\mathbb{R}^3}(I_{\alpha}*|u_n|^q)|u_n|^qdx\to 0.
        \end{align*}

        From \eqref{eq:NP functional} and the fact $J(u_n,v_n)=0$, it follows that 
        \begin{align}\label{eq:nu-big-3}
            \begin{split}
 &2(a_1\|\nabla u_n\|_2^2 + a_2\|\nabla v_n\|_2^2) + 4\left(V_1\|u_n\|_2^2 + V_2\|v_n\|_2^2- 2\lambda\int_{\mathbb{R}^3}u_nv_ndx\right) \\
 &+2(b_1\|\nabla u_n\|_2^4+b_2\|\nabla v_n\|_2^4) =\frac{12+\alpha}{3+\alpha}\mu B({\nu})+o(1)
            \end{split}
        \end{align}
By \eqref{eq:lower_critical_constant} we see that 
\begin{align*}
\begin{split}
  &\ 4V_1(1-\delta)\mathcal{S}_*B(\nu)^{\frac{3}{3+\alpha}}+o(1) \\
\le &\ 4V_1(1-\delta)\|v_n\|_2^2 \\
\le &\ 2(a_1\|\nabla u_n\|_2^2 + a_2\|\nabla v_n\|_2^2) + 4\left(V_1\|u_n\|_2^2 + V_2\|v_n\|_2^2- 2\lambda\int_{\mathbb{R}^3}uvdx\right) \\
   &\quad +2(b_1\|\nabla u_n\|_2^4+b_2\|\nabla v_n\|_2^4) \\
=&\ \frac{12+\alpha}{3+\alpha}\mu B(\nu)+o(1) \end{split}
\end{align*}
We now assume that $B(\nu)>0$. Then we obtain
\begin{align}\label{eq:B_nu_est}
\left(\frac{4(3+\alpha)V_1(1-\delta)}{\mu(12+\alpha)}\mathcal{S}_*\right)^{\frac{\alpha+3}{\alpha}}\le B(\nu). 
\end{align}
On the other hand, from \eqref{eq:B_nu_est} and \eqref{eq:functional Psi} we obtain
\begin{align*}
    &\ m_{\mu,\nu}+o(1)\\
\ge &\ \frac{1}{2}(a_1\|\nabla u_n\|_2^2+a_2\|\nabla v_n\|_2^2) \\
     &\ \ +\frac{\alpha+1}{2(3+\alpha)}\mu\int_{\mathbb{R}^3}(I_{\alpha}*|u_n|^{1+\frac{\alpha}{3}})|u_n|^{1+\frac{\alpha}{3}}dx+\frac{q+\alpha-1}{8q}\nu\int_{\mathbb{R}^3}(I_{\alpha}*|v_n|^q)|v_n|^qdx \\
     \ge&\ \frac{\alpha+1}{2(\alpha+3)}\mu\left(\frac{4(3+\alpha)V_1(1-\delta)}{\mu(12+\alpha)}\mathcal{S}_*\right)^{\frac{3+\alpha}{\alpha}}+o(1). 
\end{align*}
and 
\begin{align*}
m_{\mu,\nu}\ge \frac{\alpha+1}{2(\alpha+3)}\mu\left(\frac{4(3+\alpha)V_1(1-\delta)}{\mu(12+\alpha)}\mathcal{S}_*\right)^{\frac{3+\alpha}{\alpha}}.
\end{align*}
This is a contradiction to \eqref{eq:nu_big}. Hence $B(\nu)=0$. However, by \eqref{eq:nu-big-3} this implies that $(u_n,v_n)\to 0$ in $H$ which is contradiction to Lemma \ref{Lem:positive on NP}.   
\end{proof}
\begin{proof}[Proof of (2) of Theorem B]
Since \eqref{eq:compactness for critical-2} holds, we can consider the shift sequence $(\Tilde{u}_n(x), \Tilde{v}_n(x)) \coloneqq (u_n(x + y_n), v_n(x + y_n))$ and we see that the weak limit is nontrivial. Therefore, we can repeat the same arguments as the proof of Theorem \ref{Th:non-critical}. The proof of 
(2) of Theorem \ref{Th:critical} has been completed.
\end{proof}
\section{Proof of Theorem \ref{Th:nonexistence}}

In this section we give the proof of Theorem \ref{Th:nonexistence}. 

\begin{proof}[Proof of Theorem \ref{Th:nonexistence}]
Suppose that $(u, v)\in H$ is a weak solution to \eqref{eq:NKC}. From $\langle I'(u,v), (u,v)\rangle =0$ we have
\begin{align}\label{eq:Nehari}
\begin{split}
 &\ \langle I'(u,v), (u,v)\rangle \\
=&\ (a_1\|\nabla u\|_2^2+a_2\|\nabla v\|_2^2)+\left(V_1\|u\|_2^2+V_2\|v\|_2^2-2\lambda\int_{\mathbb{R}^3}uvdx\right) \\
&\ \ \ +(b_1\|\nabla u\|_2^4+b_2\|\nabla v\|_2^4)-\mu\int_{\mathbb{R}^3}(I_{\alpha}*|u|^p)|u|^pdx-\nu\int_{\mathbb{R}^3}(I_{\alpha}*|v|^q)|v|^qdx=0.
\end{split}
\end{align}
By Proposition \ref{Lem:Pohozaev} we also have the Pohozaev identity
\begin{align}\label{eq:Pohozaev-ThC}
\begin{split}
P(u,v)=&\ \frac{1}{2}(a_1\|\nabla u\|_2^2 + a_2\|\nabla v\|_2^2) + \frac{3}{2}\left(V_1\|u\|_2^2 + V_2\|v\|_2^2\right)+\frac{1}{2}(b_1\|\nabla u\|_2^4 + b_2\|\nabla v\|_2^4) \\
            &\quad- \frac{3+\alpha}{2p}\mu\int_{\mathbb{R}^3}(I_{\alpha}*|u|^p)|u|^pdx - \frac{3+\alpha}{2q}\nu\int_{\mathbb{R}^3}(I_{\alpha}*|v|^q)|v|^qdx\\ 
            &\quad -3 \lambda \int_{\mathbb{R}^3} uv\,dx = 0.
\end{split}
\end{align}
We first consider the case $p=q=3+\alpha$. By \eqref{eq:Nehari}, \eqref{eq:Pohozaev-ThC} and \eqref{eq:-lambda} we obtain 
\begin{align*}
0=&\ P(u,v)-\frac{1}{2}\langle I'(u,v), (u,v)\rangle \\
 =&\ \left(V_1\|u\|_2^2+V_2\|v\|_2^2-2\lambda\int_{\mathbb{R}^3} uvdx \right) \\
 \ge&\ (1-\delta)(V_1\|u\|_2^2+V_2\|v\|_2^2).
\end{align*}
Therefore $u=v=0$.

We next consider the case $p=q=\frac{3+\alpha}{3}$. By \eqref{eq:Nehari} and  \eqref{eq:Pohozaev-ThC} we obtain 
\begin{align*}
0=&\ \frac{3}{2}\langle I'(u,v), (u,v)\rangle-P(u,v) \\
 =&\ \frac{1}{2}(a_1\|\nabla u\|_2^2+a_2\|\nabla v\|_2^2)+\frac{3}{2}(b_1\|\nabla u\|_2^4+b_2\|\nabla v\|_2^4) \\
 \ge&\ \frac{1}{2}\mathcal{S}(a_1\|u\|_6^2+a_2\|v\|_6^2).
\end{align*}
Therefore $u=v=0$.

The proof of Theorem \ref{Th:nonexistence} has been completed. 
\end{proof}

\section*{Disclosure statement} 

No potential conflict of interest was reported by the author.

\end{document}